\documentclass[11pt]{article}

\usepackage[T1]{fontenc}
\usepackage[utf8]{inputenc}
\usepackage{authblk}
\usepackage[margin=1in]{geometry}
\usepackage{amsmath,amssymb,amsthm}
\usepackage[usenames,dvipsnames,svgnames,table]{xcolor}
\usepackage{graphicx}
\usepackage{stmaryrd}
\usepackage{mathrsfs}
\usepackage{enumitem}
\usepackage{subcaption}
\usepackage{bbm}
\usepackage{glossaries}
\usepackage{verbatim}
\usepackage{tcolorbox}
\usepackage{comment}
\usepackage{tabulary}
\usepackage{booktabs}
\usepackage{tikz-cd}    
\usepackage{tikz}

\newcommand{\bbN}{\mathbb{N}}

\newcommand{\bbR}{\mathbb{R}}
\newcommand{\bbS}{\mathbb{S}}

\newcommand{\bbZ}{\mathbb{Z}}

\newcommand{\bfD}{\mathbf{D}}
\newcommand{\bfE}{\mathbf{E}}
\newcommand{\bfF}{\mathbf{F}}

\newcommand{\bfI}{\mathbf{I}}

\newcommand{\bfM}{\mathbf{M}}
\newcommand{\bfN}{\mathbf{N}}
\newcommand{\bfO}{\mathbf{O}}

\newcommand{\bfS}{\mathbf{S}}

\newcommand{\bfc}{\mathbf{c}}

\newcommand{\clF}{\mathcal{F}}

\newcommand{\clH}{\mathcal{H}}
\newcommand{\clI}{\mathcal{I}}

\newcommand{\clL}{\mathcal{L}}

\newcommand{\clP}{\mathcal{P}}

\newcommand{\clR}{\mathcal{R}}
\newcommand{\clS}{\mathcal{S}}
\newcommand{\clT}{\mathcal{T}}

\newcommand{\srD}{\mathscr{D}}
\newcommand{\srE}{\mathscr{E}}
\newcommand{\srF}{\mathscr{F}}

\newcommand{\srH}{\mathscr{H}}

\newcommand{\srR}{\mathscr{R}}

\newcommand{\Tan}{\mathrm{Tan}}

\newcommand{\spt}{\mathrm{spt}}

\newcommand{\scan}{{\it Scan}}
\newcommand{\newM}{\underline{\bfM}}

\newcommand{\join}{\hbox{{$\times$}\kern-0.65em{$\times$}}}

\DeclareMathAlphabet{\pazocal}{OMS}{zplm}{m}{n}

\newcommand{\myell}{\,\hbox{\vrule height 7pt depth 0pt
		\vrule height 0.4pt depth 0pt width 6pt}\,}

\DeclareMathAlphabet{\pazocal}{OMS}{zplm}{m}{n}
\newtheorem*{thm*}{Theorem}

\usepackage{mathtools}

\newtheorem{thm}{Theorem}[section]
\newtheorem{remark}[thm]{Remark}
\newtheorem{lem}[thm]{Lemma}
\newtheorem{definition}[thm]{Definition}

\newtheorem{proposition}[thm]{Proposition}
\newtheorem{ex}[thm]{Example}

\newtheorem{corollary}[thm]{Corollary}

\usepackage{hyperref}
\usepackage{parskip}
\hypersetup{
	colorlinks=true,
	linkcolor=blue,
	filecolor=magenta,      
	urlcolor=cyan,
}

\usepackage[normalem]{ulem}

\definecolor{darkgrn}{rgb}{0, 0.75, 0}

\begin{document}
\title{Partial Plateau's Problem with $H$-mass}
\author[1]{Enrique Alvarado\thanks{ealvarado@math.ucdavis.edu}}
\author[2]{Qinglan Xia\thanks{qlxia@math.ucdavis.edu}}
\affil[1,2]{Department of Mathematics, University of California at Davis, Davis, CA 95616}

\maketitle

\abstract{Classically, Plateau's problem asks to find a surface of least area with a given boundary $B$. 
In this article, we investigate a version of Plateau's problem, where the boundary of an admissible surface is only required to partially span $B$. 
Our boundary data is given by a flat $(m-1)$-chain $B$ and a smooth compactly supported differential $(m-1)$-form $\Phi$. 
We are interested in minimizing
$
\bfM(T) - \int_{\partial T} \Phi
$
over all $m$-dimensional rectifiable currents $T$ in $\bbR^n$ such that $\partial T$ is a subcurrent of the given boundary $B$. 
The existence of a rectifiable minimizer is proven with Federer and Fleming's compactness theorem. 
We generalize this problem by replacing the mass $\bfM$ with the $H$-mass of rectifiable currents.
By minimizing over a larger class of objects, called {\it scans with boundary}, and by defining their $H$-mass as a type of lower-semicontinuous envelope over the $H$-mass of rectifiable currents, we prove an existence result for this problem by using Hardt and De Pauw's BV compactness theorem.
}


\section{Introduction}

\paragraph{A brief history of Plateau's problem.}

Loosely speaking, Plateau's problem asks to find a {\it surface} of least {\it area} with a given {\it boundary}.
The various answers to this question depend on the precise notion of a surface, of its area, and of its boundary. 
Although this problem was classically studied for surfaces in $\bbR^3$ with the use of mappings of surfaces~\cite{nitsche1989lectures}, geometric measure theory has now provided us with precise formulations and definitions of these italicized terms~\cite{federer-1969-1, simon-1984-lectures, morgan-2000-1, ambrosio-2000-2, reifenberg-1960-1}.  
Motivated by De Rham's notion of currents, in the seminal work~\cite{federer-1960-1}, Federer and Fleming solved a version of Plateau's problem with a class of generalized oriented submanifolds called {\it integral currents}, which come equipped with a natural notion of a boundary and surface area.
We refer the reader to \cite{fleming2015geometric} for a historical overview of Plateau's problem and geometric measure theory.

\paragraph{The $H$-mass and scans.}
The {\it mass} $\bfM(T)$ of a current $T$ is used as a way to measure its area, and for a rectifiable current $T = \clH^m\myell M\wedge \theta\xi$, its mass is given by $\int_M \theta\, d\clH^m$ where $\theta : M \to \bbN$ is its multiplicity function (see~\S\ref{sec: notationDefinitions}). For surfaces in $\bbR^3$, a mass minimizing current $T$ provides a good model for some but not all soap films.
For example, in~\cite{fleming1962oriented}, Fleming showed that 2-dimensional mass minimizing currents in $\bbR^3$ are smooth embedded surfaces away from their boundary, whereas general soap films have interior singularities. 
To get a better model for soap films, Almgren~\cite{almgren1986deformations} introduced the notion of {\it size} for rectifiable currents and investigated size minimizing currents. 
However, in general, the existence of size minimizing currents is difficult to obtain.
In part, this is due to the fact that size minimizing sequences may have unbounded mass, as can be seen in examples provided by F.~Morgan in~\cite{morgan1989size}.
The lack of bounded mass prevents one from applying Federer and Fleming's {\it compactness theorem}~\cite{federer-1960-1} to size minimizing sequences.

Another important measurement of surface area is the $\alpha$-mass, \begin{equation}\label{eq: alphaMass}
\bfM_\alpha(T) = \int_M \theta^\alpha\, d\clH^m 
\end{equation}
of a real rectifiable current $T = \clH^m\myell M \wedge \theta\xi$, where the parameter $\alpha \in [0, 1]$. 
The $\alpha$-mass acts as a bridge between the mass and the size of a real rectifiable current. 
When $\alpha = 0$, we get the size of $T$, and when $\alpha = 1$ we get its mass. 
Given two probability measures $\mu_+$ and $\mu_-$, any real rectifiable $1$-dimensional current $T$ with $\partial T = \mu_+ - \mu_-$ is called a {\it transport path from $\mu_+$ to $\mu_-$} in the study of branched (ramified) optimal transport~\cite{xia2003optimal, bernot2008optimal, xia2015motivations}. 
The $\alpha$-mass of $T$ is called the $\alpha$-cost of the transport path $T$.
A more general version of the $\alpha$-mass is the $H$-mass. 
For a {\it concave integrand} (Definition~\ref{def:concaveIntegrand}) $H : [0, \infty) \to [0, \infty)$, the $H$-mass of a real rectifiable current $T = \clH^m\myell M \wedge \theta\xi$ is given by $\int_M H\circ \theta\, d\clH^m$. 

In article~\cite{de2003size}, Hardt and De Pauw investigated a Plateau-type problem: minimizing the $H$-mass over all rectifiable currents with a fixed boundary.
Although $\bfM_H$ minimizing sequences of rectifiable currents may not necessarily have bounded mass, they showed that for any fixed $\epsilon > 0$, the minimizing sequences for the approximate problem: $\bfM_H + \epsilon \bfM$ do.
They then considered a sequence of rectifiable currents, each of which is a minimizer to the approximate problem for a sequence $\epsilon_i \to 0$.
Using their general BV-compactness theorem on this sequence, they obtained an $\bfM_H$ minimizing subsequence of rectifiable currents, that when viewed as scans, converges to a rectifiable scan. 
An $m$-dimensional {\it scan} in $\bbR^n$ is a type of measurable function $\clT$ from the space of $(n - m)$ planes, to the space of $0$ dimensional currents in $\bbR^n$ endowed with the {\it $H$-flat distance} (Definition~\ref{def:HFlatDistance}).
For example, an $m$-dimensional rectifiable current in $\bbR^n$ can be viewed as a scan by slicing it with $n - m$ planes (see \S\ref{sec: the hmass and the hflat distance}).

\paragraph{Partial Plateau's problem.}

\begin{figure}[b]
    \centering
    \includegraphics[width=.9\textwidth]{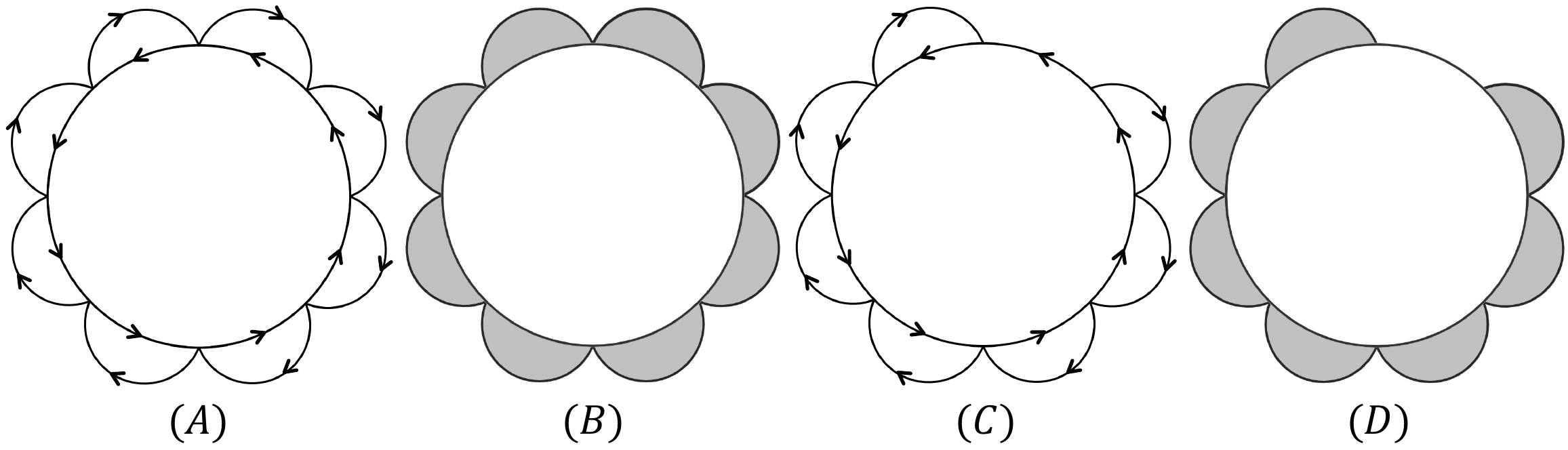}
    \caption{The solution to the standard Plateau's problem for the closed curve shown in $(A)$ is the surface shown in $(B)$. 
    When one of the outer arcs is removed as shown in $(C)$, the soap film shown in $(D)$ may still form. 
    Although it is no longer a solution to the standard Plateau's problem, it is a solution to partial Plateau's problem studied in the examples of \S\ref{section: PPPwithMass}.}
    \label{fig:sunflower}
\end{figure}

In the above Plateau-type problems, the boundaries of the surfaces are prescribed. 
In this article, we propose a new Plateau-type problem, where the boundary is not required to completely fulfill the prescribed one, but rather, is only required to satisfy a portion of it.
We call this {\it partial Plateau's problem}. 
To motivate it, let us consider the following scenario.
Suppose that we have an oriented wire that has the form of a sunflower as in Figure \ref{fig:sunflower}($A$). 
The solution to the standard oriented Plateau's problem with this boundary will be the union of all petals as in Figure \ref{fig:sunflower}($B$). 
Now, assume that one petal accidentally falls off the flower and the boundary wire becomes the one shown in Figure \ref{fig:sunflower}($C$). 
In this case, there is no solution to the corresponding Plateau's problem anymore since the resulting wire is no longer a closed curve. 
Nevertheless, the remaining wire still generates a soap film in the form of the remaining pedals as shown in Figure \ref{fig:sunflower}($D$).
This scenario (where the soap film only spans a portion of the given wire) motivates us to relax the boundary condition.
In Example~\ref{ex:1}, we investigate this example in detail, and show that the solution shown in Figure \ref{fig:sunflower}($D$) is given by a solution to our partial Plateau's problem. 

Another motivation for this problem comes from the study of ramified optimal transportation as given in \cite{xia2023ramified}.
More precisely, let $\mu $ and $\nu $ be two Radon measures on a convex compact subset $X$ of the Euclidean space $\mathbb{R}^n$, let $\mathbf{M}_{\alpha }$ be the $\alpha$-mass as defined in (\ref{eq: alphaMass}) for some $\alpha \in [0,1)$, and let $h$ be a continuous function on the support of the signed measure $\nu - \mu$.
Then, \cite{xia2023ramified} considered the following
resource allocation problem:
Minimize
\begin{equation}
\label{eqn: E}
\mathbf{E}_\alpha^h(T):=\mathbf{M}_{\alpha }(T)-\int_{X}h\,d(\partial T)
\end{equation}%
among all 1 dimensional real rectifiable currents $T$ with
$\partial T \preceq \nu - \mu$, in the sense that its Jordan decomposition $\partial T = \tilde{\nu} - \tilde{\mu}$ satisfies $\tilde{\mu} \leq \mu$ and $\tilde{\nu} \leq \nu$. 
In \cite{xia2023ramified}, this problem was motivated by considering the following example: Consider a firm that produces and sells a product in
various regions.
The measures $\mu $
and $\nu $ represent, respectively, the distributions of production
capacities and market sizes. 
The function $h$ represents the sale price on the support of $\nu$, and the production cost on the support of $\mu$.
The firm aims to maximize its profit (the negative of $\mathbf{E}_\alpha^h(T)$) defined as sale revenue $\int_X h\, d\tilde{\nu}$ minus the cost
involved in transportation $\bfM_\alpha(T)$ and production $\int_X h\,d\tilde{\mu}$.
Note that the boundary $\partial T$ is not required to equal $\nu -\mu $, but rather to be a portion of it. 
Similar kinds of optimal partial mass transportation problems have been studied for instance by Caffarelli and McCann \cite{caffarelli2010free} and also Figalli \cite{figalli2010optimal} for the scenario of Monge-Kantorovich problems with particular attention to the quadratic cost. 

\paragraph{Partial Plateau's problem with $H$-mass.}

The above Plateau-type problems either look to minimize functionals involving 
\begin{itemize}
\item the {\it mass}, or more generally, the {\it $H$-mass} over $m$-dimensional {\it rectifiable} currents with a {\it fixed} boundary condition, or
\item the {\it $\alpha$-mass} over {\it 1 dimensional real rectifiable} currents with a {\it partial} boundary condition. 
\end{itemize}
In this article, we investigate a Plateau type problem that minimizes a functional involving the {\it $H$-mass} over {\it $m$-dimensional rectifiable} currents with a {\it partial} boundary condition.
Precisely, let $H$ be a concave integrand, let $B$ be an $(m - 1)$-dimensional rectifiable current with finite $H$-mass, and let $\Phi$ be a smooth, compactly supported $(m - 1)$-differential form in $\bbR^n$.
We are interested in minimizing
\[
\bfE_{H, \Phi}(T) := \bfM_H(T) - \int_{\partial T} \Phi
\]
over all $m$-dimensional rectifiable currents $T$ in $\bbR^n$ with $\partial T \preceq B$. 
By definition, $\partial T \preceq B$ means that $\partial T$ is a {\it subcurrent}~\cite{paolini2006optimal} of $B$, in the sense that $\bfM(B) = \bfM(B - \partial T) + \bfM(\partial T)$. 
For the sake of intuition, it is worth mentioning that whenever $B$ is a real rectifiable current with finite mass, $\partial T \preceq B$ means that $\partial T$ is supported within the support of $B$, has the same orientation as $B$, but only has a fraction of its density~\cite[\text{Lemma 3.7}]{paolini2006optimal}. 

We have thus far formulated our Plateau type problems as minimization problems over rectifiable currents.
However, to obtain our main existence results (Theorem~\ref{thm: plateausProblemWithHMass}, Theorem~\ref{thm: mainthm}, and Corollary~\ref{cor: main}), we perform our minimization over a larger class of objects, called {\it $m$-dimensional scans with boundary} (see Definitions~\ref{defn: scan} and~\ref{def: boundary}).
Loosely speaking, these are {\it scans} (introduced in~\cite{de2003size}) with well-defined boundaries that are obtained as pointwise limits of rectifiable currents (when viewed as scans).
The reason that we minimize over scans with boundary is that when viewed as scans, the class of $m$-dimensional rectifiable currents is not closed under pointwise almost everywhere convergence.

\paragraph{Organization of the article.}
In \S\ref{sec: notationDefinitions}, we go over relevant definitions and notations that are by now standard in geometric measure theory literature. 
Then in \S\ref{section: PPPwithMass}, we look at {\it partial Plateau's problem with mass}: Given an $(m - 1)$-dimensional flat chain $B$ in $\bbR^n$ with finite mass, and a smooth, compactly supported $(m - 1)$-differential form $\Phi$ in $\bbR^n$, consider
\begin{equation}\label{eq: PPwithMass}
\clP_{(\bfE_\Phi,\, \preceq B,\, \bbR^n)}
    \begin{cases}
          \text{minimize } \bfE_\Phi(T) := \bfM(T) - \partial T(\Phi)\\
          \text{among } T \in \srR_m(\bbR^n) \text{ such that } \partial T \preceq B.
     \end{cases}
\end{equation}
In Theorem~\ref{thm: existenceE_Phi}, we prove the existence of this problem by using the standard compactness arguments of rectifiable currents.
Here, for any $\bfE_\Phi$ minimizing sequence $\{T_i\}$, we have uniform bounds for $\{\bfM(T_i) + \bfM(\partial T_i)\}$.
Although this is a special case of partial Plateau's problem with $H$-mass (studied in~\S\ref{sec: PPPwithHmass}) when $H(\theta) = \theta$, we do not have to minimize over $m$-dimensional scans in $\bbR^n$ as in the following problems that involve the general $H$-mass.

In \S\ref{sec: plateausProblemWithHmass}, we re-investigate {\it Plateau's problem with $H$-mass}, which was first investigated by De Pauw and Hardt in their seminal work~\cite{de2003size}.
Formulated over rectifiable currents, the problem is stated as follows: Given a rectifiable current $T_0 \in \srR_m(\bbR^n)$, 
\begin{equation}\label{eq: HPP-currents}
\begin{cases}
          \text{minimize } \bfM_H(T) \\
          \text{among } T \in \srR_m(\bbR^n) \text{ such that } \partial T = \partial T_0.
     \end{cases}
\end{equation}
To get an existence-type result~\cite[Theorem 3.5.2.]{de2003size}~for problem (\ref{eq: HPP-currents}), De Pauw and Hardt innovatively relax the notion of a rectifiable current to a rectifiable scan. 
In their proof of the theorem, they construct a sequence of $\bfM_H$ minimizing rectifiable currents that pointwise converge (as scans) to a rectifiable scan $\clT^\ast$. 
By Fateau's lemma, it is clear that $\bfM_H(\clT^\ast)$ is less than or equal to the infimum value of problem (\ref{eq: HPP-currents}). 
Nevertheless, since $\clT^\ast$ is not necessarily a rectifiable current, one cannot conclude that $\bfM_H(\clT^\ast)$ achieves the infimum value of problem (\ref{eq: HPP-currents}) as claimed in~\cite[Theorem 3.5.2.]{de2003size}.

To overcome this issue, in \S\ref{sec:4.3} we amend their existence result~\cite[Theorem 3.5.2.]{de2003size}~by reformulating problem (\ref{eq: HPP-currents}) with a modified version of the $H$-mass, denoted by $\newM_H$, and by minimizing $\newM_H$ over $Scan_m(\bbR^n)$, the class of {\it $m$-dimensional scans with boundary} (Definitions~\ref{defn: scan} and~\ref{def: boundary}).
In particular, we provide the following reformulation of {\it Plateau's problem with $H$-mass}: Given a scan $\clT_0 \in \scan_m(\bbR^n)$ with boundary,
\begin{equation}\label{eq: HPP-scans}
\clP_{(\newM_H,\, \partial \clT_0,\, \bbR^n)} 
    \begin{cases}
          \text{minimize } \newM_H(\clT) \\
          \text{among } \clT \in Scan_m(\bbR^n) \text{ such that } \partial\clT = \partial \clT_0.
     \end{cases}
\end{equation}
It is worth mentioning the following two facts about $Scan_m(\bbR^n)$. 
The class of $m$-dimensional scans with boundary is larger than the class of $m$-dimensional rectifiable currents, and any element in $Scan_m(\bbR^n)$ has a well-defined boundary, which is not necessarily the case for the general $m$-dimensional scans studied in~\cite{de2003size}. 
Under such modifications, in Theorem~\ref{thm: plateausProblemWithHMass} we are able to get the existence of an $\newM_H$ minimizing $m$-dimensional scan with fixed boundary. 

Given $B$, $\Phi$, and $H$ as above, and the additional constraint that $B$ is rectifiable and has finite $H$-mass, in \S\ref{sec: PPPwithHmass} we consider partial Plateau's problem with $H$-mass over rectifiable currents: 
\begin{equation}\label{problem: 4}
    \begin{cases}
        \text{minimize } \bfE_{\Phi,H}(T) := \bfM_H(T) - \partial T(\Phi)\\
          \text{among } T \in \srR_m(\bbR^n) \text{ such that } \partial T \preceq B.
     \end{cases}
\end{equation}
Unfortunately, as in Plateau's problem (\ref{eq: HPP-currents}) with $H$-mass, this problem does not always have a solution among rectifiable currents. 
To resolve this issue, we reformulate the problem by using our modified version $\newM_H$ of the $H$-mass on the class $Scan_m(\bbR^n)$, and investigate our main problem of interest, called {\it partial Plateau's problem with $H$-mass}:
\begin{equation}\label{eq:PPPwithHmass}
\clP_{(\bfE_{\Phi,H},\,\preceq B,\, \bbR^n)}
    \begin{cases}
          \text{minimize } \bfE_{\Phi,H}(\clT) := \newM_H(\clT) - \partial \clT(\Phi)\\
          \text{among } \clT \in Scan_m(\bbR^n) \text{ such that } \partial \clT \preceq \clS(B).
     \end{cases}
\end{equation}

Our main existence result, Theorem~\ref{thm: mainthm}, is proven within a slightly more general setting, and the existence of minimizer $\clT^\ast$ to problem (\ref{eq:PPPwithHmass}) is then given in Corollary~\ref{cor: main}.
Moreover, Corollary~\ref{cor: main} also shows us that
\[
\bfE_{\Phi,H}(\clT^*) = \inf\{\bfE_{\Phi,H}(T) : T \in \srR_m(\bbR^n), \partial T \preceq B\}.
\]
In other words, despite the fact that one can not expect to find an $\bfE_{H, \Phi}$ minimizer to problem (\ref{problem: 4}) among rectifiable currents, its infimum value is obtained by the corresponding $\bfE_{H, \Phi}$ minimizer of problem (\ref{eq:PPPwithHmass}), where we minimize over scans with boundary under the same partial boundary condition.

\section{Notation and preliminary definitions}\label{sec: notationDefinitions}

\begin{table}[htbp]\caption{Notation}
\centering 
\begin{tabular}{r c p{12cm} }
\toprule
$\bbN$ & $:=$ & \{1, 2, 3, \dots\}, the set of all natural numbers\\
$H$ & & A concave integrand \hfill (Def.~\ref{def:concaveIntegrand})\\
$\bfO^\ast(n, k)$ & & the space of all orthogonal projections of $\bbR^n$ onto $\bbR^k$ \\
$\boldsymbol{\theta}_{n, k}^\ast$ & &  the $\bfO(n)$ invariant measure on $\bfO^\ast(n, k)$ \\
$\clL^k$ & &  the Lebesgue measure on $\bbR^k$ \\
$\clH^k$ & &  the $k$ dimensional Hausdorff measure \\
\multicolumn{3}{c}{}\\
\multicolumn{3}{c}{\underline{Currents}}\\
\multicolumn{3}{c}{}\\
$\srD_m(\bbR^n)$ & & $m$-dimensional currents in $\bbR^n$ \\
$\srE_m(\bbR^n)$ & & compactly supported $m$-dimensional currents in $\bbR^n$\\
$\bfM(T)$ & & the mass of $T \in \srD_m(\bbR^n)$ \hfill (\ref{defn: mass})\\
$\srR_m(\bbR^n)$ & & $m$-dimensional rectifiable currents in $\bbR^n$ \\
$\bfF_K(T)$ & & flat norm of $T \in \srD_m(\bbR^n)$ \hfill (\ref{eq:flatNorm1})\\
$\bfI_{m}(\bbR^n)$ & & $m$-dimensional integral currents in $\bbR^n$ \hfill (\ref{eq: integralCurrentsIntegralFlatChains})\\
$\srF_m(\bbR^n)$ & & $m$-dimensional integral flat chains in $\bbR^n$ \hfill (\ref{eq: integralCurrentsIntegralFlatChains})\\
$\bfN_m(\bbR^n)$ & & $m$-dimensional normal currents in $\bbR^n$ \hfill (\ref{eq: normalCurrentsFlatChains})\\
$\bfF_m(\bbR^n)$ & & $m$-dimensional flat chains in $\bbR^n$ \hfill (\ref{eq: normalCurrentsFlatChains})\\
$\bfM_H(T)$ & & $H$-mass of $T \in \srD_m(\bbR^n)$ \hfill (Def.~\ref{def:Hmass})\\
$\srF_K^H(T)$ & & $H$-flat distance of $T \in \srD_m(\bbR^n)$ \hfill (Def.~\ref{def:HFlatDistance})\\
$A \preceq B$ & & $A$ is a subcurrent of $B$ for $A, B \in \srD_m(\bbR^n)$ \hfill (\ref{def:subcurrent})\\
\multicolumn{3}{c}{}\\
\multicolumn{3}{c}{\underline{Scans}}\\
\multicolumn{3}{c}{}\\
$\clS(T)$ & & the measurable map $(p, y) \mapsto \langle T, p, y\rangle$ associated to a flat chain $T$ \hfill~(\ref{eq: scan of a current})\\
$\bfM_H(f)$ & & the $H$-mass of a measurable map $f: \bfO^\ast(n, m) \times \bbR^m \to \bfI_{0, K}(\bbR^n)$ \hfill (\ref{eqn: M_H_scan})\\
$\clS_m(\bbR^n)$ & & scans obtained from rectifiable currents in $\bbR^n$ \hfill ~(\ref{eq: scansFromRectifiableCurrents})\\
$\scan_m(\bbR^n)$ & & $m$-dimensional scans in $\bbR^n$ with boundary \hfill (Def.~\ref{defn: scan} and~\ref{def: boundary})\\
$\newM_{H}(\clT)$ & & the lower-semicontinuous envelope of $\bfM_H$ over a.e.~pointwise convergence of $\{\clS(R) : R\in \srR_m(\bbR^n), \partial \clT  = \partial R\}$ \hfill (Def.~\ref{def: L_H})\\
\bottomrule
\end{tabular}
\label{tab:TableOfNotation}
\end{table}

In this article, most of our notations and definitions are consistent with that of Federer's geometric measure theory book~\cite{federer-1969-1} along with De Pauw and Hardt's article~\cite{de2003size}. 
For the convenience of the reader, we will now give a brief review of the standard concepts and results from geometric measure theory that will be important for us in this article.

Let $0 \leq m \leq n$.
We let $\srD^m(\bbR^n)$ denote the real vector space of smooth differential $m$-forms with compact support. 
The class $\srD_m(\bbR^n)$ is the vector space of all continuous real-valued linear functions on $\srD^m(\bbR^n)$. Elements of $\srD_m(\bbR^n)$ are called {\it $m$-dimensional currents} in $\bbR^n$, and the subspace $\srE_m(\bbR^n)$ of $\srD_m(\bbR^n)$ contains all $m$-dimensional currents with compact support.
The {\it mass} of a current $T \in \srD_m(\bbR^n)$ is defined as 
\begin{equation}\label{defn: mass}
\bfM(T) := \sup \{T(\phi) : \phi \in \srD^m(\bbR^n),\text{ with } \langle e_1\wedge\dots\wedge e_m, \phi(x)\rangle \leq 1,\  \forall x \in \bbR^n \text{ and } e_i \in \bbS^{n-1}\}.
\end{equation}
For $m \geq 1$, the {\it boundary} of a current $T \in \srD_m(\bbR^n)$ is the current $\partial T \in \srD_{m - 1}(\bbR^n)$ defined by $\partial T(\phi) := T(d\phi)$ for $\phi \in \srD^{m-1}(\bbR^n)$, and $\partial T$ is defined to be the zero-current whenever $T \in \srD_0(\bbR^n)$. 
For a sequence $\{T_i\}_{i=1}^\infty\subseteq \srD_m(\bbR^n)$ of $m$-dimensional currents and a current $T \in \srD_m(\bbR^n)$, we say $T_i$ {\it converges} (weakly) to $T$, written as $T_i \rightharpoonup T$, if
$T_i(\phi) \to T(\phi)$ for all $\phi \in \srD^m(\bbR^n)$.

The rectifiable currents of Federer and Fleming play an important role in this article. 
With $\clH^m$ denoting the $m$-dimensional Hausdorff measure on $\bbR^n$, a subset $M$ of $\bbR^n$ is called {\it $(\clH^m, m)$-rectifiable} if $\clH^m(M) <\infty$ and $\clH^m(M\setminus \cup_{i = 1}^\infty N_i) = 0$ for some finite or countable family $\{N_i\}_{i \in I}$ of $m$-dimensional $C^1$-submanifolds of $\bbR^n$. 
The following is an important characterization of the class $\srR_m(\bbR^n)$ of $m$-dimensional rectifiable currents.

Let $0 \leq m \leq n$. 
An {\it $m$-dimensional rectifiable current} $T \in \srR_m(\bbR^n)$ is given by the following three things:
\begin{enumerate}
\item an $\clH^m$-measurable and bounded $(\clH^m, m)$-rectifiable subset $M$ of $\bbR^n$, 
\item an $\clH^m$-measurable $m$-vectorfield $\xi : M \to \wedge_m\bbR^n$ such that for $\clH^m$-a.e.~$x \in M$, $\xi(x) = e_1\wedge\dots\wedge e_m$ for some orthonormal basis $\{e_i\}_{i = 1}^m$ of the approximate tangent space $\Tan(M, x)$,
\item and an $(\clH^m\myell M)$-summable function $\theta : M \to \bbN$ called the {\it multiplicity function}. 
\end{enumerate}

We will be denoting such a rectifiable current $T \in \srR_m(\bbR^n)$ by $\clH^m\myell M \wedge \theta\xi$ as in~\cite{federer-1969-1}.
Other common notation is also given by $\underline{\underline{\tau}}(M, \theta, \xi)$ as in \cite{simon-1984-lectures}, or by $\llbracket M, \xi, \theta\rrbracket$. 
Its action on $\phi \in \srD^m(U)$ is given by 
\[
(\clH^m\myell M \wedge \theta\xi)(\phi) = \int_{M} \langle \phi, \xi\rangle \theta\, d\clH^m,
\]
and its mass has the simple form
\begin{equation}\label{eq: massOfARectifiableCurrent}
\bfM(\clH^m\myell M \wedge \theta\xi) = \int_M \theta\, d\clH^m.
\end{equation}
\begin{remark}
Notice that the $m$-dimensional rectifiable currents we are considering all have finite mass and compact support.
\end{remark}

We will now give a brief overview of some special classes of currents. 
Whenever $K \subseteq \bbR^n$ is a {\it compact Lipschitz neighborhood retract} (CLNR)~\cite[4.1.29.]{federer-1969-1}, (e.g., compact convex subsets of $\bbR^n$), we have the identity
\begin{equation}\label{eq: compactRectifiableCurrents}
\srR_{m, K}(\bbR^n) = \srR_m(\bbR^n)\cap \{T : \spt(T) \subseteq K\}.
\end{equation}
When a subset $K$ is simply a compact subset of $\bbR^n$, $\srR_{m, K}(\bbR^n)$ has a more technical definition. 
However, since we will mainly be dealing with CLNRs of $\bbR^n$, we will forgo stating its definition here, and instead refer the reader to~\cite[4.1.24.]{federer-1969-1}.

Let $K$ be a compact subset of $\bbR^n$.
We define the following classes of currents 
\begin{itemize} 
\item $\bfI_{0, K}(\bbR^n) := \srR_{0, K}(\bbR^n)$, and $\bfI_{m, K}(\bbR^n) := \{T : T \in \srR_{m, K}(\bbR^n),\ \partial T \in \srR_{m - 1, K}(\bbR^n)\}$ for $m > 0$;
\item $\srF_{m, K}(\bbR^n) := \{R + \partial S : R \in \srR_{m, K}(\bbR^n),\ S \in \srR_{m + 1, K}(\bbR^n)\}$;
\item $\bfN_{m,K}(\bbR^n) := \{T\in \srD_m(\bbR^n) : \bfM(T) + \bfM(\partial T) <\infty \text{ with }\spt(T)\subseteq K\}$;
\item $\bfF_{m, K}(\bbR^n) :=$ the $\bfF_K$ closure of $\bfN_{m, K}(\bbR^n)$ in $\srD_m(\bbR^n)$, where $\bfF_K : \srD_m(\bbR^n) \to [0, \infty]$ is the {\it flat norm} defined as 
\begin{equation}\label{eq:flatNorm1}
\bfF_K(T) := \inf\{\bfM(T - \partial S) + \bfM(S) : S \in \srD_{m + 1}(\bbR^n)\text{ with } \spt(S)\subseteq K\}, \forall T\in \srD_m(\bbR^n).
\end{equation}
\end{itemize}

\begin{remark}\label{rem: 0dimIntegralCurrents}
Any $0$ dimensional integral current $T \in \bfI_{0, K}(\bbR^n)$ has a representation  
\[
T = \sum_{i = 1}^k a_i \boldsymbol{\delta}_{x_i} 
\]
where $a_1, \dots, a_k$ are positive integers, and $x_1, \dots, x_k \in K$. 
\end{remark}

The members of the abelian groups
\begin{align}
\bfI_m(\bbR^n) &:= \bigcup_{K\subseteq \bbR^n\text{ compact}} \bfI_{m, K}(\bbR^n),  &\srF_m(\bbR^n) := \bigcup_{K\subseteq \bbR^n\text{ compact}} \srF_{m, K}(\bbR^n),\label{eq: integralCurrentsIntegralFlatChains}\\
\bfN_m(\bbR^n) &:= \bigcup_{K\subseteq \bbR^n\text{ compact}} \bfN_{m, K}(\bbR^n),  &\bfF_m(\bbR^n) := \bigcup_{K\subseteq \bbR^n\text{ compact}} \bfF_{m, K}(\bbR^n)\label{eq: normalCurrentsFlatChains}
\end{align}
are called $m$-dimensional {\it integral currents, integral flat chains, normal currents}, and {\it flat chains} in $\bbR^n$, respectively.
One can show the following set-theoretic relations between the classes of currents that we have defined so far: 
\[
\begin{tikzcd}
\bfI_m(\bbR^n)
\arrow[d, phantom, sloped, "\subseteq"] \arrow[r, phantom, sloped, "\subseteq"] &\srR_m(\bbR^n) \arrow[d, phantom, sloped, "\subseteq"] \arrow[r, phantom, sloped, "\subseteq"]& \srF_m(\bbR^n) \arrow[d, phantom, sloped, "\subseteq"]\\
\bfN_m(\bbR^n) \arrow[r, phantom, sloped, "\subseteq"] & \{T\in\bfF_m(\bbR^n) : \bfM(T) <\infty\}\arrow[r, phantom, sloped, "\subseteq"]  &\bfF_m(\bbR^n)\arrow[r, phantom, sloped, "\subseteq"] & \srD_m(\bbR^n).
\end{tikzcd}
\]
\begin{remark}
The flat norm endows $\srD_m(\bbR^n)$ with a metric defined by $(T, T') \mapsto \bfF_K(T - T')$. 
By~\cite[4.2.18.]{federer-1969-1}, one can metrize integral flat chains $\srF_{m, K}(\bbR^n)$ with
\begin{equation}\label{eq:flatNorm3}
\srF_K(T) := \min\{\bfM(R) + \bfM(S) : R \in \srR_{m, K}(\bbR^n),\text{ and }\ S \in \srR_{m+1, K}(\bbR^n) \text{ with } T = R + \partial S\}.
\end{equation}
\end{remark}

\begin{remark}\label{rem: weakConvergenceAndFlatNorm}
For any sequence $\{T_i\}_{i = 1}^\infty \subseteq \srD_m(\bbR^n)$ of currents, and a current $T \in \srD_m(\bbR^n)$, if $\lim_{i\to \infty}\bfF_K(T_i - T) = 0$ for some compact subset $K$ of $\bbR^n$, then $T_i \rightharpoonup T$. 
Conversely, if $\{T_i\}_{i = 1}^\infty \subseteq \bfI_m(\bbR^n)$ is a sequence of integral currents with $\sup_i \bfM(T_i) + \bfM(\partial T_i) < \infty$, then $T_i \rightharpoonup T$ implies $\lim_{i \to \infty}\bfF_K(T_i - T) = 0$ for all compact subsets $K$ of $\bbR^n$. 
We refer the reader to Leon Simon's book on geometric measure theory~\cite[Theorem 31.2]{simon-1984-lectures} for a proof of this fact.
\end{remark}

\section{Partial Plateau's problem with mass}\label{section: PPPwithMass}
As in~\cite{paolini2006optimal}, for any two $k$ dimensional currents $A$ and $B$ of finite mass, we say that $A$ is a {\it subcurrent} of $B$, denoted by $A \preceq B$, if 
\begin{equation}\label{def:subcurrent}
\bfM(B) = \bfM(B - A) + \bfM(A).
\end{equation}
Notice that the zero-current $\boldsymbol{0}$ is always a subcurrent of $B$ and that $\bfM(A) \leq \bfM(B)$ whenever $A \preceq B$. In particular, $A\preceq {\bf 0}$ if and only if $A={\bf 0}$.

\begin{remark}\label{rem: subcurrent}
Let $B = \srH^m\myell M \wedge \theta \xi$ be a real-rectifiable current.
If $A$ is any $m$-dimensional subcurrent of $B$, Paolini and Stepanov~\cite[\text{Lemma 3.7}]{paolini2006optimal} showed that $A = \srH^m\myell M \wedge \lambda \theta\xi$ for some Borel function $\lambda : \bbR^n \to [0, 1]$. 
In particular, $A$ is a real-rectifiable current with $\spt(A) \subseteq \spt(B)$. 
\end{remark}

\begin{lem}\label{lemma:subcurrentandconvergence}
Let $B$ be a $k$ dimensional current and $\{A_i\}_{i = 1}^\infty$ be a sequence of $k$ dimensional currents such that $A_i \preceq B$ for all $i \in \bbN$. 
If $A$ is an $k$ dimensional current such that $A_i  \rightharpoonup A$, then $A \preceq B$. 
\end{lem}

\begin{proof}
Since $\bfM$ satisfies the triangle inequality, it is sufficient to show that $\bfM(B - A) + \bfM(A) \leq \bfM(B)$.
By lower-semicontinuity of $\bfM$ with respect to weak convergence, 
\[\bfM(B-A) \leq \liminf_{i\to \infty} \bfM(B-A_i) \quad \text{ and }\quad \bfM(A) \leq \liminf_{i\to \infty} \bfM(A_i).\]
Therefore since $A_i \preceq B$ for all $i \in \bbN$,
\begin{align*}
\bfM(B - A) + \bfM(A) &\leq \liminf_{i \to \infty}  \bfM(B-A_i)+\liminf_{i \to \infty}  \bfM(A_i)\\
 &\leq \liminf_{i \to \infty} \left(\bfM(B - A_i) + \bfM(A_i)\right)\leq \bfM(B).
\end{align*}
\end{proof}

\begin{remark}\label{rem: massLscWithFlatConvergence}
We mention that Lemma~\ref{lemma:subcurrentandconvergence} holds true if we replace $A_i \rightharpoonup A$ by $\bfF_K(A_i-A)\rightarrow 0$ for some compact subset $K$ of $\bbR^n$.
This is because by Remark~\ref{rem: weakConvergenceAndFlatNorm}, $\bfF_K(A_i-A)\rightarrow 0$ implies $A_i \rightharpoonup A$.
\end{remark}

We are interested in what we call, {\it partial Plateau's problem}: 
Given a flat chain $B \in \bfF_{m-1}(\bbR^n)$ with mass $\bfM(B) <\infty$, a smooth compactly supported form $\Phi \in \srD^{m - 1}(\bbR^n)$, and a rectifiable current $T_0 \in \srR_m(\bbR^n)$, consider 
\[
\clP_{(\bfE_{\Phi},  \preceq B,  T_0, \bbR^n)}
    \begin{cases}
          \text{minimize } \bfE_{\Phi}(T) \\
          \text{among } T \in \srR_m(\bbR^n) \text{ such that } \partial (T-T_0) \preceq B,
     \end{cases}
\]
where  
\begin{equation}
\label{eq: E_Phi}
    \bfE_{ \Phi}(T) := \bfM(T) - \partial T(\Phi).
\end{equation}

\begin{remark}
Note the following: 
\begin{itemize}
    \item When $B={\bf 0}$, $\partial(T - T_0) \preceq \boldsymbol{0}$ implies $\partial T = \partial T_0$. 
    Thus, up to a constant, this problem corresponds to the standard Plateau's problem of rectifiable currents: Minimize $\bfE_{\Phi}(T):=\bfM(T) - \partial T_0(\Phi)$ among all $T \in \srR_m(\bbR^n)$ with $\partial T=\partial T_0$. 
    \item The introduction of $T_0$ makes this a slightly more general problem than $\clP_{(\bfE_\Phi,\, \preceq B,\, \bbR^n)}$, given in~(\ref{eq: PPwithMass}). 
    When $T_0={\bf 0}$, the boundary of $T$ is a subcurrent of $B$.
    \item We require $\bfM(B)< \infty$ but do not require $\bfM(\partial T_0)< \infty$.
    Thus for any admissible $T$, the mass of its boundary $\bfM(\partial T)$ may be unbounded.
\end{itemize}
\end{remark}

We now state the existence result of the partial Plateau's problem.
\begin{thm}\label{thm: existenceE_Phi}
Let $\Phi \in \srD^{m - 1}(\bbR^n)$ and $B \in \bfF_{m - 1}(\bbR^n)$ with $\bfM(B) < \infty$. 
Given $T_0 \in \srR_m(\bbR^n)$,
there exists a rectifiable current $T^\ast \in \srR_m(\bbR^n)$ with $\partial (T^\ast- T_0)\preceq B$ such that 
\[
\bfE_{\Phi}(T^\ast) = \min \{\bfE_{\Phi}(T): T \in \srR_m(\bbR^n),
\partial (T-T_0)\preceq B\}.
\]
\end{thm}

\begin{proof}
Notice that, since the zero-current ${\bf 0}\preceq B$, this class contains $T_0$. 
By definition, both $T_0$ and $B$ have compact support.
Let $K$ be a compact convex subset of $\bbR^n$ containing $\spt(T_0)\cup \spt(B)$ and let $\pi_K : \bbR^n \to \bbR^n$ denote the nearest point projection onto $K$~\cite[4.1.15]{federer-1969-1}.

For any $T \in \srR_m(\bbR^n)$ with $\partial(T - T_0) \preceq B$, it holds that $\bfE_\Phi((\pi_K)_\# T) \leq \bfE_\Phi(T)$. 
Indeed, observe that $\spt(\partial T) \subseteq K$ because $\partial T = \partial(T - T_0) + \partial T_0$, $\spt(\partial(T - T_0)) \subseteq \spt(B)$ (see Remark~\ref{rem: subcurrent}), and $\spt(T_0)\cup \spt(B) \subseteq K$.  
Thus, $\partial \left((\pi_K)_\# T\right)=(\pi_K)_\#\partial T=\partial T$. Since also
 $\bfM((\pi_K)_\# T) \leq \bfM(T)$, it follows that
\begin{equation}
\label{eqn: pi_K}
\bfE_\Phi((\pi_K)_\# T)=\bfM((\pi_K)_\# T)-\partial ((\pi_K)_\# T)(\Phi)\le\bfM(T)-\partial T(\Phi)=\bfE_\Phi(T).
\end{equation} 
Let $\{T_i\}_{i = 1}^\infty\subseteq \srR_m(\bbR^n)$ be an $\bfE_{\Phi}$-minimizing sequence with $\partial (T_i-T_0)\preceq B$.  
Thus, for each $i$,
\[\bfM(\partial (T_i-T_0))\leq\bfM(B)<\infty.\]  By (\ref{eqn: pi_K}), without loss of generality we may assume $\spt(T_i) \subseteq K$ and $\bfE_{\Phi}(T_i)\le \bfE_{\Phi}(T_0)$. Then,
\[
\bfM(T_i) = \bfE_{\Phi}(T_i) + \partial T_i (\Phi)\leq \bfE_{\Phi}(T_0)+ \partial T_i (\Phi)=\bfM(T_0)+ \partial (T_i-T_0) (\Phi)\leq \bfM(T_0)+  \bfM(B)||\Phi||,
\]
where $||\Phi||$ is the {\it comass} norm of $\Phi$, and 
\[
\bfM(T_i-T_0) \leq \bfM(T_i)+\bfM(T_0)\leq 2\bfM(T_0)+  \bfM(B)||\Phi||<\infty.
\]
Thus, $\{T_i - T_0\} \subseteq \bfI_{m, K}(\bbR^n) \cap \{T : \bfN(T) < c\}$ for some large enough $c > 0$. 
By the Federer and Fleming's compactness theorem for integral currents~\cite[4.2.17(b)]{federer-1969-1}, a subsequence of $\{T_i-T_0\}$, still denoted by $\{T_i - T_0\}$, $\srF_K$ converges to some $R\in \bfI_{m, K}(\bbR^n)$. 
Thus, $T_i - T_0 \rightharpoonup R$ and $T_i \rightharpoonup T^\ast := R + T_0 \in \srR_m(\bbR^n)$.
Since $\partial (T_i-T_0)\preceq B$, and $\partial(T_i - T_0) \rightharpoonup \partial R = \partial (T^\ast - T_0)$, by Lemma \ref{lemma:subcurrentandconvergence}, $\partial (T^\ast-T_0)\preceq B$ as well. 
Also, by the lower-semicontinuity of $\bfM$,
\[\bfE_{\Phi}(T^\ast)=\bfM(T^\ast)-\partial T^\ast(\Phi)\leq \liminf_{i\rightarrow \infty}\bfM(T_i)-\partial T_i(\Phi)=\liminf_{i\rightarrow \infty}\bfE_{\Phi}(T_i).\]
Since $\{T_i\}$ is an $\bfE_{\Phi}$-minimizing sequence, we have $T^\ast$ is the desired $\bfE_{\Phi}$-minimizer.
\end{proof}

\begin{ex}\label{ex:1}
We now investigate the sunflower example as shown in Figure~\ref{fig:sunflower} with more detail. 
Let $B$ be the 1 dimensional rectifiable current as shown in Figure~\ref{fig:flowerdata}(A) with density 2 on the inner circle, and density 1 on the outer circular arcs.
Let $\llbracket P_i\rrbracket$ (for $i = 1, \dots, 8$) and $\llbracket D\rrbracket$ denote the 2 dimensional rectifiable currents with clockwise orientation which respectively represent the petals and the central disk of the sunflower, as shown in Figure~\ref{fig:flowerdata}(B).
Then any $T \in \srR_2(\bbR^2)$ with $\partial T \preceq B$ has the form
\begin{equation}\label{eq: candidate}
T = a \llbracket D\rrbracket + \sum_{i = 1}^8 c_i \llbracket P_i\rrbracket 
\end{equation}
where $a$ and $c_i$ are integers such that 
\begin{equation}\label{eq: example-inequality}
0 \leq c_i \leq 1 \quad\text{ and } \quad 0 \leq c_i - a \leq 2\quad \text{ for each } i = 1, 2, \dots, 8.
\end{equation}
\begin{figure}[h]
    \centering
    \includegraphics[width=.75\textwidth]{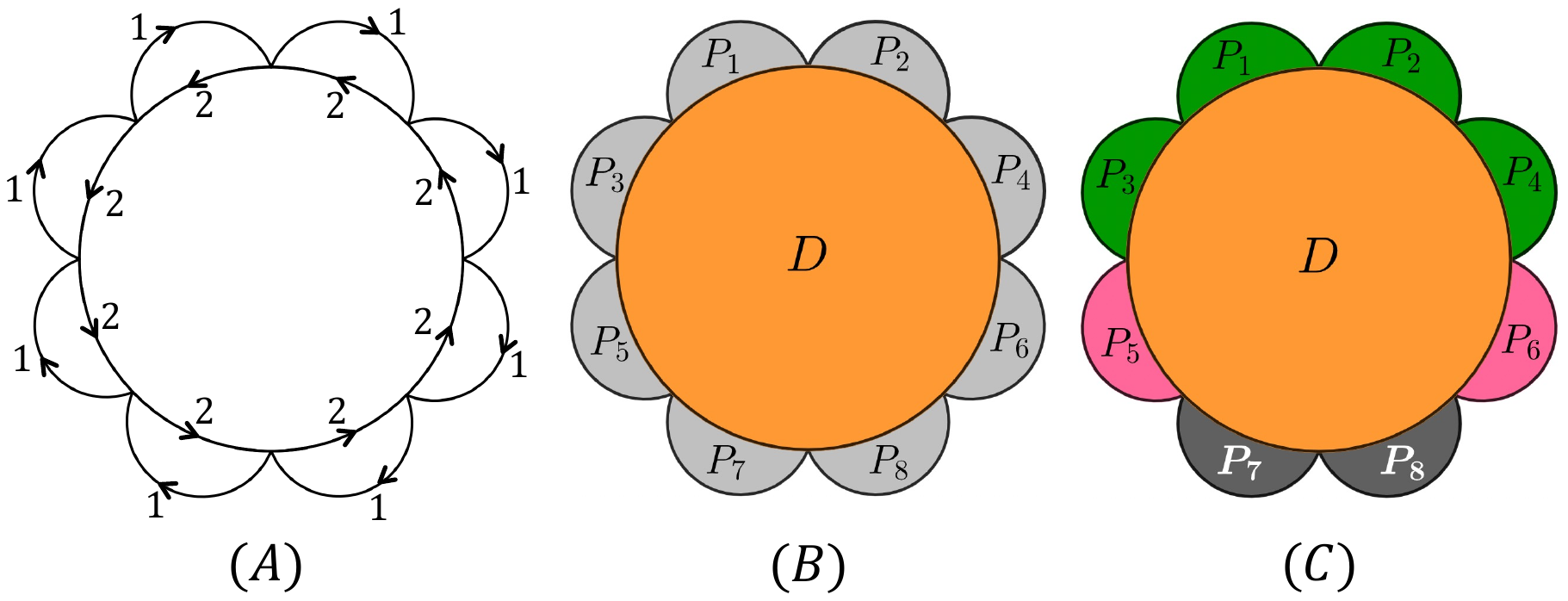}
    \caption{The boundary $B$ is shown in $(A)$ along with its orientation and density. 
    The petals $\{P_i\}_{i = 1}^8$ and central disk $D$ shown in $(B)$ have a clockwise orientation. 
    When $\Phi$ is given, the petals are partitioned into the three families shown in (C): negative petals (shown in green), neutral petals (shown in pink), and positive petals (shown in gray).}
    \label{fig:flowerdata}
\end{figure}

If $\Phi = 0$, then clearly we get that the $0$-current $\boldsymbol{0}$ is the unique minimizer. 
For any smooth, compactly supported $1$-form $\Phi \in \srD^1(\bbR^2)$, 
\begin{align*}
\bfE_\Phi(T) &= \bfM(T) - \partial T(\Phi)\\
&= |a|\bfM(\llbracket D\rrbracket) + \sum_{i = 1}^8 |c_i| \bfM(\llbracket P_i\rrbracket) - \left(a\partial \llbracket D\rrbracket(\Phi) + \sum_{i = 1}^8 c_i\partial \llbracket P_i\rrbracket(\Phi) \right)\\
&= |a|\bfM(\llbracket D\rrbracket) - a\partial \llbracket D\rrbracket(\Phi) + \sum_{i = 1}^8 c_i( \bfM(\llbracket P_i\rrbracket) - \partial \llbracket P_i\rrbracket(\Phi)).
\end{align*}
We now investigate the possible solutions for $\clP_{(\bfE_\Phi, \preceq B, \boldsymbol{0}, \bbR^2)}$. 
First, partition the collection $\{\llbracket P_i\rrbracket \}$ of petals into the following sets
\begin{align*}
\clP_- &= \{\llbracket P_i\rrbracket : \bfM(\llbracket P_i \rrbracket) - \partial \llbracket P_i \rrbracket (\Phi) < 0\} \text{ of negative petals,}\\
\clP_0 &= \{\llbracket P_i\rrbracket : \bfM(\llbracket P_i \rrbracket) - \partial \llbracket P_i \rrbracket (\Phi) = 0\} \text{ of neutral petals, and}\\
\clP_+ &= \{\llbracket P_i \rrbracket: \bfM(\llbracket P_i \rrbracket) - \partial \llbracket P_i \rrbracket (\Phi) > 0\} \text{ of positive petals}.
\end{align*}
In Figure~\ref{fig:flowerdata}(C), the set of negative petals $\{\llbracket P_1\rrbracket, \llbracket P_2\rrbracket, \llbracket P_3\rrbracket, \llbracket P_4\rrbracket\}$ are colored green, the neutral petals $\{\llbracket P_5\rrbracket, \llbracket P_6\rrbracket\}$ are colored pink, and the positive petals $\{\llbracket P_7\rrbracket, \llbracket P_8\rrbracket\}$ are colored gray.

Any  possible solution $T$ for $\clP_{(\bfE_\Phi, \preceq B, \boldsymbol{0}, \bbR^2)}$ must have the form described in (\ref{eq: candidate}). 
From inequalities (\ref{eq: example-inequality}), we get that $a \in \{-2, -1, 0, 1\}$. 
We describe the possible solutions: 
\begin{itemize}
\item When $a=-2$, all the $c_i$'s must be $0$. 
In this case, the only possible solution is
\[
T_{-2}=-2\llbracket D \rrbracket
\]
as illustrated in Figure~\ref{fig:sunflower2}(A). 

\item When $a = -1$ all $c_i$ are allowed to be  either $0$ or $1$ for any admissible candidate. 
However, for a solution $T_{-1}$ with $a = -1$, the coefficients $c_i$ must be $1$ on the negative petals and $0$ on the positive petals.  
On the neutral petals, $c_i$ may be chosen to be either $0$ or $1$. 
This is because removing any positive petal or adding any negative petal from a candidate minimizer can only decrease $\bfE_{\Phi}$, while removing or adding a neutral petal does not change the value of $\bfE_{\Phi}$. 
That is, any possible solution must be in the form 
\[
T_{-1} = -\llbracket D\rrbracket + \sum_{P_i \in \clP_-} \llbracket P_i\rrbracket + \sum_{P_i \in \clP_0^\ast} \llbracket P_i\rrbracket \quad \text{ for any } \clP_0^\ast \subseteq \clP_0,
\] 
as illustrated in Figure~\ref{fig:sunflower2}(B).

\item When $a = 0$, for the same reason as in the case of $a = -1$, any possible solution must be in the form 
\[
T_0 = \sum_{P_i \in \clP_-} \llbracket P_i\rrbracket + \sum_{P_i \in \tilde{\clP_0}} \llbracket P_i\rrbracket \quad \text{ for any } \tilde{\clP_0} \subseteq \clP_0,
\]
as illustrated in Figure~\ref{fig:sunflower2}(C).

\item When $a = 1$, $c_i = 1$ for all $i = 1, \dots, 8$. 
Hence, the only possible solution is 
\[
T_1 = \llbracket D\rrbracket + \sum_{i = 1}^8 \llbracket P_i \rrbracket, 
\]
as illustrated in Figure~\ref{fig:sunflower2}(D).
\end{itemize}
\begin{figure}[h]
    \centering
    \includegraphics[width=.8\textwidth]{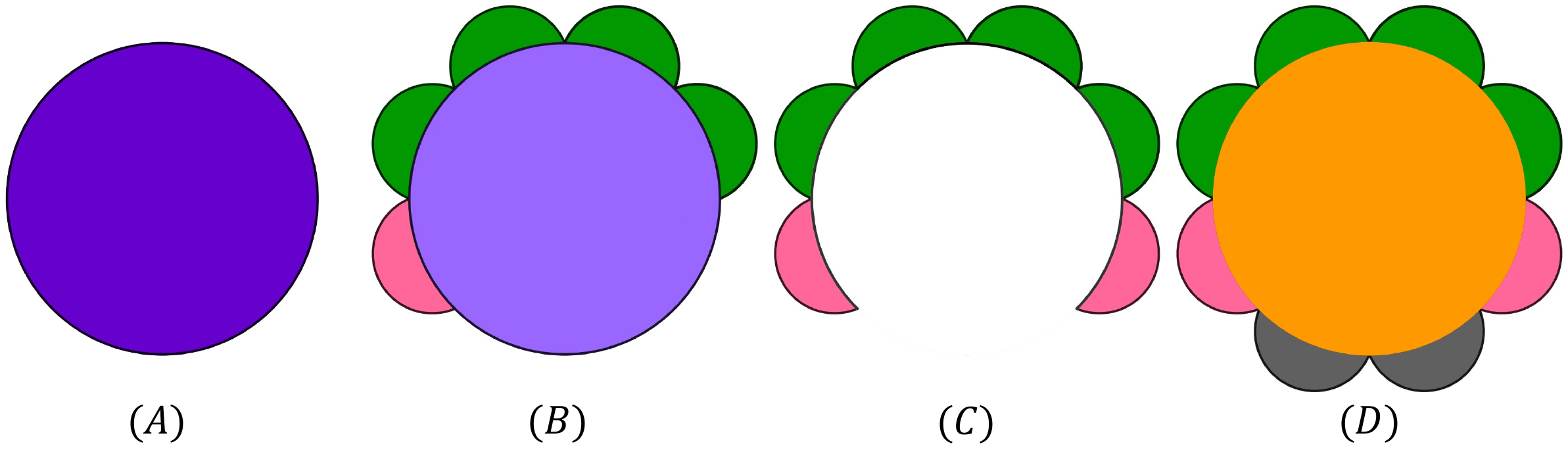}
    \caption{From left to right, we show the four possible solutions $T_{-2}, T_{-1}, T_0, T_1$ for $\clP_{(\bfE_\Phi, \preceq B, \boldsymbol{0}, \bbR^2)}$ with boundary $B$ and the 1-form $\Phi$ as shown in Figure~\ref{fig:flowerdata}.}
    \label{fig:sunflower2}
\end{figure}
So far, we have described the possible solutions of $\clP_{(\bfE_\Phi, \preceq B, \boldsymbol{0}, \bbR^2)}$ by only taking into account the subcurrent restrictions on their boundaries and the way that the negative, neutral, and positive petals change the value of $\bfE_\Phi$. 
We will now characterize the solutions for $\clP_{(\bfE_\Phi, \preceq B, \boldsymbol{0}, \bbR^2)}$ in terms of how $\partial \llbracket D\rrbracket$ acts on $\Phi$. 
First note the following values of $\bfE_\Phi(T_i)$ for $i = -2, -1, 0, 1$:
\begin{align*}
&\bfE_\Phi(T_{-2}) = 2\bfM(\llbracket D\rrbracket) + 2\partial \llbracket D\rrbracket(\Phi)\\
&\bfE_\Phi(T_{-1}) = \bfM(\llbracket D\rrbracket) + \partial \llbracket D\rrbracket (\Phi) + \sum_{P_i \in \clP_-}( \bfM(\llbracket P_i\rrbracket) - \partial \llbracket P_i\rrbracket(\Phi))\\
&\bfE_\Phi(T_{0}) = \sum_{P_i \in \clP_-} (\bfM(\llbracket P_i\rrbracket) - \partial \llbracket P_i\rrbracket (\Phi))\\
&\bfE_\Phi(T_{1}) = \bfM(\llbracket D\rrbracket) - \partial \llbracket D\rrbracket (\Phi) + \sum_{P_i \in \clP_-}( \bfM(\llbracket P_i\rrbracket) - \partial \llbracket P_i\rrbracket(\Phi)) + \sum_{P_i \in \clP_+} ( \bfM(\llbracket P_i\rrbracket) - \partial \llbracket P_i\rrbracket(\Phi)).
\end{align*}
Thus,
\begin{itemize}
\item $\bfE_\Phi(T_{-2})\le \bfE_\Phi(T_{-1})$ whenever 
\[
\partial \llbracket D\rrbracket(\Phi) \leq  
 -\bfM(\llbracket D\rrbracket) + \sum_{P_i \in \clP_-}( \bfM(\llbracket P_i\rrbracket) - \partial \llbracket P_i\rrbracket(\Phi)) =: \lambda_{-2},
 \] 
 \item $\bfE_\Phi(T_{-1})\le  \bfE_\Phi(T_0) $ whenever 
 \[ \partial \llbracket D\rrbracket(\Phi) \leq -\bfM(\llbracket D\rrbracket) =: \lambda_{-1}, \text{ and }\]
 \item $\bfE_\Phi(T_0) \le \bfE_\Phi(T_1) $ whenever 
 \[\partial \llbracket D\rrbracket(\Phi) \leq \bfM(\llbracket D\rrbracket) + \sum_{P_i \in \clP_+} \left(\bfM(\llbracket P_i\rrbracket) - \partial \llbracket P_i\rrbracket (\Phi)\right) =:  \lambda_0\]
 \end{itemize}
Notice also that we have the following inequalities 
\[\lambda_{-2}\le \lambda_{-1}\le \lambda_{0}.\]
As a result, 
\begin{itemize}
\item $T_{-2}$ is a solution when $ \partial \llbracket D\rrbracket(\Phi) \leq 
 \lambda_{-2}$, 
 \item $T_{-1}$ is a solution when $ \partial \llbracket D\rrbracket(\Phi) \in [ \lambda_{-2}, \lambda_{-1}]
 $, 
 \item $T_{0}$ is a solution when $ \partial \llbracket D\rrbracket(\Phi) \in [ \lambda_{-1}, \lambda_{0}]$, and
 \item $T_{1}$ is a solution when $ \partial \llbracket D\rrbracket(\Phi) \geq
 \lambda_0$.
 \end{itemize}
 In addition, the solution is unique whenever $\partial\llbracket D \rrbracket (\Phi) \notin \{\lambda_{-2}, \lambda_{-1}, \lambda_0, \lambda_1\}$.
\end{ex}

\begin{figure}[b]
    \centering
    \includegraphics[width=0.9\textwidth]{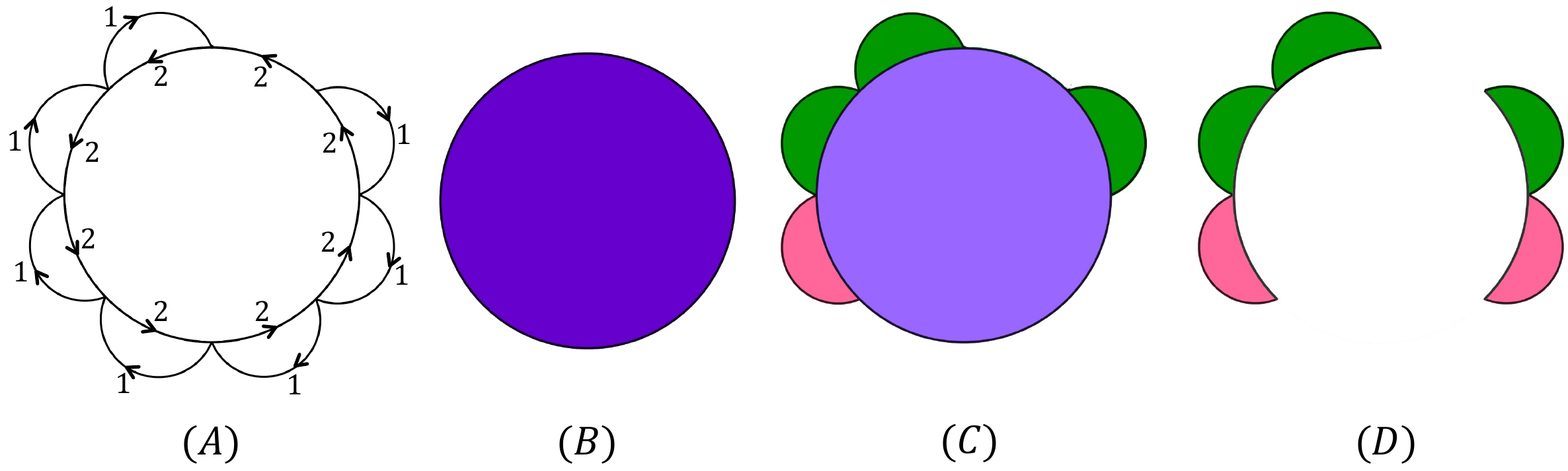}
    \caption{Given the boundary $\tilde{B}$ as shown in $(A)$, the three possible solutions $T_{-2}, T_{-1}, T_0$ for $\clP_{(\bfE_\Phi, \preceq \tilde{B}, \boldsymbol{0}, \bbR^2)}$ are shown in $(B)$, $(C)$, and $(D)$, respectively.}
    \label{fig:sunflower4}
\end{figure}

\begin{ex}\label{ex:2}
Let $\tilde{B}$ be the $1$ dimensional rectifiable current shown in Figure~\ref{fig:sunflower4}(A).
As in Example \ref{ex:1}, $T_{-2}$ is a solution when $\partial \llbracket D\rrbracket (\Phi) \leq \lambda_{-2}$,  $T_{-1}$ is a solution whenever $\partial \llbracket D\rrbracket (\Phi) \in [\lambda_{-2}, \lambda_{-1}]$, and $T_0$ is a solution when $\partial \llbracket D\rrbracket (\Phi) \geq \lambda_{-1}$.
However, $T_1$ cannot be a solution for any $\Phi$ since its boundary $\partial T_1$ is no longer a subcurrent of $\tilde{B}$. 
\end{ex}

\section{Plateau's problem with $H$-mass for scans with boundary}\label{sec: plateausProblemWithHmass}
Building upon the work of De Pauw and Hardt~\cite{de2003size}, in this section, we further investigate Plateau's problem with $H$-mass. 
We define a subfamily of measurable maps from the space of all $n - m$ planes in $\bbR^n$ to the space of 0 dimensional integral currents, called {\it scans with boundary} and modify the definition of the $H$-mass to prove the existence of a solution of Plateau's problem with $H$-mass for scans with boundary in Theorem~\ref{thm: plateausProblemWithHMass}. 
The definitions and techniques from this section will also be used when we study {\it partial} Plateau's problem in~\S\ref{sec: PPPwithHmass}.

\subsection{The $H$-mass and the $H$-flat distance}\label{sec: the hmass and the hflat distance}

We now list a few concepts and results from \cite{de2003size} that are particularly useful for the analysis in this article. 
We start by defining the $H$-mass of real rectifiable currents with finite size and by defining the $H$-flat distance on integral flat chains.

\begin{definition}[{\cite[Definition 3.2.1.]{de2003size}}]\label{def:concaveIntegrand}
A function $H : [0, \infty) \to [0, \infty)$ is called a concave integrand if it satisfies the following conditions:
\begin{enumerate}
\item $H(0) = 0$;
\item $H(1) = 1$;
\item $H(\theta_1) < H(\theta_2)$ whenever $0 \leq \theta_1 < \theta_2$;
\item $H(\theta_1 + \theta_2) \leq H(\theta_1) + H(\theta_2)$ whenever $\theta_1, \theta_2 \geq 0$;
\item $\lim_{\theta \to \infty} H(\theta) = \infty$. 
\end{enumerate}
\end{definition}

\begin{definition}[{\cite[Definition 3.2.2.]{de2003size}}]\label{def:Hmass}
For a concave integrand $H$ and a real rectifiable current $T \in \srD_m(\bbR^n)$ with $\bfS(T) < \infty$, the $H$-mass of $T$ is defined as
\begin{equation}\label{def:HMass}
\bfM_H(T) := \int_{\bbR^n} H(\Theta^m(\|T\|, x))\, d\clH^m(x).
\end{equation}
\end{definition}

\begin{definition}[{\cite[Definition 3.2.8.]{de2003size}}]\label{def:HFlatDistance}
For a compact subset $K$ of $\bbR^n$, a concave integrand $H$, and an integral flat chain $T \in \srF_{m, K}(\bbR^n)$ we define
\begin{equation}\label{eq:HFlatDistance}
\srF^H_K(T) := \inf\{\bfM_H(T - \partial S) + \bfM_H(S) : S \in \srR_{m + 1, K}(\bbR^n),\ T - \partial S \in \srR_{m, K}(\bbR^n)\}.
\end{equation}
\end{definition}

As in~\cite[1.7.4, 2.7.16]{federer-1969-1}, let $\bfO(n)$ be the orthogonal group of $\bbR^n$ and $\bfO^\ast(n, k)$ be the space of all orthogonal projections of $\bbR^n$ to $\bbR^k$ with its $\bfO(n)$ invariant measure $\theta^\ast_{n, k}$.
In {\cite[Proposition 3.1.3.~and 3.1.5.]{de2003size}}, it shows that any flat chain $T \in \bfF_{m, K}(\bbR^n)$ and $k \in \{1, \dots, m\}$ can be represented by the $\boldsymbol{\theta}^\ast_{n, k}\times \clL^k$ measurable map 
\begin{equation}\label{eq: scan of a current}
\clS(T) : \bfO^\ast(n, k)\times \bbR^k \to \bfF_{m - k, K}(\bbR^n), \quad \clS(T)(p, y) := \langle T, p, y\rangle,
\end{equation}
defined for every $p \in \bfO^\ast(n, k)$ and for $\clL^k$ almost every $y \in \bbR^k$.
The $k$ dimensional flat chain $\langle T, p, y\rangle$ is obtained by slicing $T$ with $p^{-1}(y)$ under the orthogonal projection map $p$ for $y \in \bbR^m$.
For the remainder of the paper, the map $\clS(T)$ is to be defined as in~(\ref{eq: scan of a current}) for $k = m$.

Let $T\in \srR_{m, K}(\bbR^n)$.
For $\clL^k$ almost every $y \in \bbR^m$, $\langle T, p, y\rangle \in \bfI_{0, K}(\bbR^n)$~\cite[4.3.6.]{federer-1969-1}, and hence $\clS(T): \bfO^\ast(n, m)\times \bbR^m \to \bfI_{0, K}(\bbR^n)$.
In addition, as in {\cite[(17)]{de2003size}}, the integral-geometric identity says that its mass can be given by integrating the mass of its $0$ dimensional slices:
\[ 
\bfM(T)=\frac{1}{\boldsymbol{\beta}_1(m, n)}\int_{\bfO^\ast(n, m)\times \bbR^m} \bfM(\clS(T)(p,y))\, d(\mathbf{\theta}_{n, m}^\ast \times \clL^m)(p, y),
\]
where $\boldsymbol{\beta}_1(m, n)$ is the constant given in {\cite[2.7.16]{federer-1969-1}}. 
By~\cite[(21)]{de2003size}, we also have a similar equality for the $H$-mass of a rectifiable current $T \in \srR_m(\bbR^n)$: 
\begin{equation}
\label{eqn: M_H_T}
    \bfM_H(T) = \frac{1}{\boldsymbol{\beta}_1(m, n)}\int_{\bfO^\ast(n, m)\times \bbR^m} \bfM_H(\clS(T)(p,y))\, d(\mathbf{\theta}_{n, m}^\ast \times \clL^m)(p, y).
\end{equation}

Motivated by~(\ref{eqn: M_H_T}), De Pauw and Hardt~\cite{de2003size} defined the $H$ mass of a {\it rectifiable scan}.
They also introduced the concept of a {\it scan cycle} from the observation that a flat chain $T \in \bfF_m(\bbR^n)$ having boundary zero is equivalent to a corresponding condition on its zero dimensional slices~\cite[Proposition 3.1.6]{de2003size}. 
Here, we extend their definitions to measurable maps from $\bfO^\ast(n, m)\times \bbR^m$ to $\bfI_{0, K}(\bbR^n)$ in the following definition.
\begin{definition}\label{def: HmassVariation}
For any $\mathbf{\theta}_{n, m}^\ast \times \clL^m$ measurable map $f : \bfO^\ast(n, m)\times \bbR^m \to \bfI_{0, K}(\bbR^n)$, we define the $H$-mass of $f$ to be
\begin{equation}
\label{eqn: M_H_scan}
    \bfM_H(f) :=\frac{1}{ \boldsymbol{\beta}_1(m, n)}\int_{\bfO^\ast(n, m)\times \bbR^m} \bfM_H(f(p,y))\, d(\mathbf{\theta}_{n, m}^\ast \times \clL^m)(p, y)
\end{equation}
and define the variation of $f$ as
\begin{equation}
\label{eqn: D_scan}
    \bfD(f) := \sup\{\int_{\bfO^\ast(n, m)\times \bbR^m} |D(\phi \circ f)|:\ \phi : \bfI_{0, K}(\bbR^n) \to \bbR \text{ with } \mathrm{Lip}(\phi) \leq 1\}.
\end{equation}
Moreover we say that $f$ is a cycle, and write $\boldsymbol{\partial}f = 0$, if for $\mathbf{\theta}_{n, m}^\ast \times \clL^m$ almost every $(p, y) \in \bfO^\ast(n, m)\times \bbR^m$,
\[
f(p, y)(\boldsymbol{1}) = 0,
\]
where $\boldsymbol{1}$ denotes the constant function $\boldsymbol{1}(x) = 1$ on $\bbR^n$. 
\end{definition}

From equations~(\ref{eqn: M_H_T}) and~(\ref{eqn: M_H_scan}), we can see that for any rectifiable current $T \in \srR_m(\bbR^n)$,
\begin{equation}\label{eq: MHTequalsMHScan}
\bfM_H(T) = \bfM_H(\clS(T)).
\end{equation}
For any compact $K\subseteq \bbR^n$, we will denote the class of measurable maps obtained by rectifiable currents in $\clR_{m, K}(\bbR^n)$ as $\clS_{m, K}(\bbR^n) := \{\clS(T) : T \in \clR_{m,K}(\bbR^n)\}$ and let 
\begin{equation}\label{eq: scansFromRectifiableCurrents}
\clS_m(\bbR^n) := \bigcup_{K\subseteq \bbR^n \text{ compact}} \clS_{m, K}(\bbR^n).
\end{equation}

We now recall one of the most important theorems from~\cite{de2003size}, and then state a direct corollary that we will use in the proofs of Theorems~\ref{thm: plateausProblemWithHMass} and~\ref{thm: mainthm}.
\begin{thm}[{\cite[Theorem 3.4.1]{de2003size}}]
Suppose $X$ is a $k$ dimensional Riemannian manifold, $Y$ is a weakly separable metric space, $M : Y\to \bbR^+$ is lower semicontinuous, and $M^{-1}([0, R])$ is sequentially compact in $Y$ for all $R > 0$. 
If $f_j : X \to Y$ is measurable, and 
\[\int_X M(f_j(x))\, d\clH^k x + \int_X |D(\phi \circ f_j)| \leq \Lambda < \infty,\]
for all $j = 1, 2, \dots$ and maps $\phi : Y \to \bbR$ with $\mathrm{Lip}(\phi) \leq 1$, then some subsequence $f_{j^\ast}$ converges pointwise $\clH^k$ a.e. to a function $f : X \to Y$ with 
\[\int_X M(f(x))\,d\clH^k x + \int_X |D(\phi \circ f)| \leq \Lambda\]
for all such $\phi$.
\end{thm}

\begin{corollary}\label{thm:compactness}
Let $f_j : \bfO^\ast(n, m)\times \bbR^m \to \bfI_{0, K}(\bbR^n)$ be a sequence of $\boldsymbol{\theta}^\ast_{n, m}\times \clL^m$ measurable maps with 
\[\bfM_H(f_j)+\bfD(f_j) \leq \Lambda < \infty,\]
for all $j = 1, 2, \dots$. 
Then there exists a subsequence $(f_{j_k})$ of $(f_j)$ and an $\boldsymbol{\theta}^\ast_{n, m}\times \clL^m$ measurable map $f : \bfO^\ast(n, m)\times \bbR^m \to \bfI_{0, K}(\bbR^n)$ such that 
\[\bfM_H(f)+\bfD(f) \leq \Lambda,\] and
$(f_{j_k})$ converges pointwise a.e.~to $f$ in the sense that for $\boldsymbol{\theta}^\ast_{n, m}\times \clL^m$ almost every $(p, y)\in \bfO^\ast(n, m)\times \bbR^m$,
\begin{equation*}
\srF_K^H[f_{j_k}(p, y) - f(p, y)]\to 0 \text{ as } k\to \infty.
\end{equation*}
\end{corollary}
\begin{proof}
This is a special case of \cite[Theorem 3.4.1]{de2003size} with
\begin{align*}
X &= \mathrm{the\ Riemannian\ manifold}\ \bfO^\ast(n, m)\times \bbR^m\\
Y &= \bfI_{0, K}(\bbR^n)\ \mathrm{with}\ \mathrm{dist}_Y(T, \tilde{T}) = \srF_K^H(T - \tilde{T})\\
M &= \bfM_H.
\end{align*}
\end{proof} 

\subsection{Existence of $H$-mass minimizing scans}\label{sec:4.3}
In this subsection, we amend the seminal work~\cite[Theorem 3.5.2.]{de2003size} of De Pauw and Hardt, who consider the following Plateau-type problem: Given $T_0 \in \srR_m(\bbR^n)$,
\[
\begin{cases}
          \text{minimize } \bfM_H(T) \\
          \text{among } T \in \srR_m(\bbR^n) \text{ such that } \partial T = \partial T_0,
     \end{cases}
\]
where $\bfM_H(T)$ denotes the $H$-mass of $T \in \srR_m(\bbR^n)$ as given in~(\ref{def:HMass}).
Let $\Gamma(\bfM_H, T_0, \mathbb{R}^n)$ be the infimum of that problem:
\begin{equation}\label{prob:DPH_Hmass_plateau}
\Gamma(\bfM_H, T_0, \mathbb{R}^n) := \inf\{\bfM_H(T) : T \in \srR_m(\bbR^n)\text{ and }  \partial T = \partial T_0\}.
\end{equation}
For the convenience of the reader, we now state~\cite[Theorem 3.5.2.]{de2003size} of De Pauw and Hardt.

\begin{thm*}[{\cite[Theorem 3.5.2.]{de2003size}}]
Let $T_0 \in \srR_m(\bbR^n)$ with $\clI_1^m(\spt(\partial T_0)) = 0$ and let $H$ be a concave integrand.
Then there exists an $m$-dimensional rectifiable scan $\clT$ in $\bbR^n$ such that $\boldsymbol{\partial}(\clT - \clS(T_0)) = 0$ and $\bfM_H(\clT) = \Gamma(\bfM_H, T_0, \mathbb{R}^n)$. 
Moreover, if $\spt(\partial T_0)$ is an $m - 1$ dimensional compact propertly embedded $C^{1, 1}$ submanifold then there exists $T \in \bfI_m(\bbR^n)$ with $\partial T = \partial T_0$ and $\bfM_H(T) = \Gamma(\bfM_H, T_0, \mathbb{R}^n)$. 
\end{thm*}

Before we discuss their proof, with the notation in~(\ref{eq: scansFromRectifiableCurrents}) and by equation~(\ref{eq: MHTequalsMHScan}), observe that~(\ref{prob:DPH_Hmass_plateau}) has the following equivalent formulation:
\begin{equation*}
\Gamma(\bfM_H, T_0, \mathbb{R}^n) = \inf\{\bfM_H(\clR) : \clR \in \clS_m(\bbR^n)\text{ and }  \boldsymbol{\partial}( \clR - \clS(T_0)) = 0\}.
\end{equation*}
In other words, $\clS_m(\bbR^n)$ is the class of objects that~\cite{de2003size} minimizes $\bfM_H$ over.

Now, in the proof, De Pauw and Hardt construct a sequence of $\bfM_H$ minimizing rectifiable currents, and use their BV compactness theorem \cite[Theorem 3.4.1.]{de2003size} to obtain a subsequence $\{T_j\}$ such that $\clS(T_j)$ converges {\it pointwise almost everywhere} to a {\it rectifiable scan} $\clT$.
By the lower-semicontinuity of $\bfM_H$ for $0$ dimensional integral currents (under the $H$-flat distance), and by Fatou's Lemma, it follows that
\[
\bfM_H(\clT)\le \Gamma(\bfM_H, T_0, \mathbb{R}^n).
\]
However, since the rectifiable scan $\clT$ is not necessarily a rectifiable current, we cannot conclude that $\clT \in \clS_m(\bbR^n)$, and in turn, can neither conclude that
\[\bfM_H(\clT) = \Gamma(\bfM_H, T_0, \mathbb{R}^n).\]
To overcome this issue (the lack of closedness of $\clS_m(\bbR^n)$ under pointwise almost everywhere convergence), we introduce a modified class of objects that we will minimize over, denoted by $\scan_m(\bbR^n)$, as well as a modified version of the $H$-mass.

\begin{definition}\label{defn: scan}
Whenever $K$ is a compact subset of $\bbR^n$, we say that an $\boldsymbol{\theta}^\ast_{n, m}\times \clL^m$ measurable map $\clT : \bfO^\ast(n, m)\times \bbR^m \to \srF_{0, K}(\bbR^n)$ is an {\it $m$-dimensional scan in $K$} if
it is in the closure of $\clS_{m,K}(\bbR^n)$. 
That is, if there exists a sequence of rectifiable currents $T_1, T_2, \dots \in \srR_{m, K}(\bbR^n)$ such that $\clS(T_i) \to \clT$ pointwise almost everywhere, in the sense that 
\begin{equation}\label{eq:0dimHFlatConvergence2}
\lim_{i\to \infty}\srF_K^H[\langle T_i, p, y\rangle - \clT(p, y)]= 0 \text{ for } \boldsymbol{\theta}^\ast_{n, m}\times \clL^m \text{ almost every } (p, y)\in \bfO^\ast(n, m)\times \bbR^m.
\end{equation}
We define 
\[
\scan_{m, K}(\bbR^n)
\]
to be the collection of all $m$-dimensional scans in $K$ such that there exists a rectifiable current $T\in \srR_{m,K}(\bbR^n)$ with $\boldsymbol{\partial}(\clT - \clS(T)) = 0$.\footnote{By Definition~\ref{def: HmassVariation}, $\boldsymbol{\partial}(\clT - \clS(T)) = 0$ means that $(\clT - \clS(T))(p, y)(\boldsymbol{1}) = 0$ for $\theta^\ast_{n, m}\times \clL^m$ almost every $(p, y) \in \bfO^\ast(n,m)\times\bbR^m$.}
Members of $\scan_{m, K}(\bbR^n)$ are called {\it $m$-dimensional scans in $K$ with boundary}. 
\end{definition}

Let us now investigate a couple of facts concerning the classes $\scan_{m, K}(\bbR^n)$ and its members. 

\begin{proposition}
Let $K_1, K_2$ be two compact subsets of $\bbR^n$.  
If $K_1\subseteq K_2$, then ${\it Scan}_{m,K_1}(\bbR^n)\subseteq {\it Scan}_{m,K_2}(\bbR^n)$.
\end{proposition}
\begin{proof}
For any $\clT\in \scan_{m, K_1}(\bbR^n)$, by definition, $\clT : \bfO^\ast(n, m)\times \bbR^m \to \srF_{0, K_1}(\bbR^n)$ is an $m$-dimensional scan in $K_1$ and $\boldsymbol{\partial}(\clT - \clS(T)) = 0$ for some $T\in \srR_{m,K_1}(\bbR^n)$.
Since $K_1\subseteq K_2$, $\clT$ is also an $m$-dimensional scan in $K_2$ and $T\in \srR_{m,K_2}(\bbR^n)$. 
So $\clT\in \scan_{m, K_2}(\bbR^n)$ as well.
\end{proof}

\begin{lem}
\label{lem: integral_flat_chains_are_scans}
Let $K$ be a compact subset of $\bbR^n$. 
For any integral flat chain $T \in \srF_{m, K}(\bbR^n)$, the $\boldsymbol{\theta}^\ast_{n, m}\times \clL^m$ measurable map 
\[
\clS(T) : \bfO^\ast(n, m)\times \bbR^m \to \srF_{0, K}(\bbR^n), \quad \clS(T)(p, y) := \langle T, p, y\rangle,
\]
is an $m$-dimensional scan in $K$ with boundary. 
\end{lem}

\begin{proof}
We first show that $\clS(T)$ is an $m$-dimensional scan in $K$. Indeed, since $T \in \srF_{m, K}$ is an integral flat chain, there exists a sequence of rectifiable currents $T_1, T_2, \cdots \in \srR_{m, K}(\bbR^n)$ such that $\lim_{i\to \infty}\clF_K(T - T_i) = 0$.
By~\cite[Remark 3.2.9.]{de2003size}, $\clF_K^H(T - T_i) \leq H(2)\clF_K(T - T_i)$, and hence $\lim_{i\to\infty} \clF_K^H(T - T_i) = 0$. 
By~\cite[Remark 3.2.11.]{de2003size}, $\int_{\bbR^m} \clF_K^H(\langle T - T_i, p, y\rangle)\,d\clL^m y \leq \clF_K^H(T - T_i)$ for any $p \in \bfO^\ast(n, m)$.
Thus, 
\[
\int_{\bfO^\ast(n, m)\times \bbR^m} \clF_K^H(\langle T - T_i, p, y\rangle)\,d(\boldsymbol{\theta}^\ast_{n, m}\times \clL^m)(p, y) \leq \boldsymbol{\beta}_1(n, m) \clF_K^H(T - T_i) \to 0 \quad \text{ as } i\to \infty. 
\]
By~\cite[Theorem 1.21]{evans-2015-measure}, there exists a subsequence $(T_{ij})$ of $(T_i)$ such that 
\[
\lim_{j\to \infty}\clF_K^H(\langle T - T_{ij}, p, y\rangle ) = 0
\]
for $\boldsymbol{\theta}^\ast_{n, m}\times \clL^m$ almost every $(p, y)\in \bfO^\ast(n, m)\times \bbR^m$. 
As a result, $\clS(T)$ is an $m$-dimensional scan in $K$.

We now show that $\clS(T)$ is also in $\scan_{m,K}(\bbR^n)$.
Since $T\in \srF_{m, K}(\bbR^n)$, there exists $R\in \srR_{m,K}(\bbR^n)$ and $Q\in \srR_{m+1,K}(\bbR^n)$ such that $T=R+\partial Q$. 
Thus,
by ~\cite[Proposition 3.1.6.]{de2003size},
\begin{equation}
    \boldsymbol{\partial}(\clS(T) - \clS(R))=\boldsymbol{\partial}(\clS(T-R)) = \boldsymbol{\partial}(\clS(\partial Q))=0.
    \label{eqn: boundary_S(T)_S(R)}
\end{equation}
Therefore, $\clS(T)\in\scan_{m,K}(\bbR^n)$.
\end{proof}

\begin{definition}\label{def: boundary}
We define the class
\[
\scan_m(\bbR^n) := \bigcup_{K\subseteq \bbR^n \text{ compact}} \scan_{m, K}(\bbR^n),
\]
whose members are called {\it $m$-dimensional scans in $\bbR^n$ with boundary}.
For each $\clT \in \scan_{m}(\bbR^n)$, the boundary of $\clT$ is defined by 
\begin{equation}
\label{eqn: defn_boundary_of_scan}
    \partial \clT := \clS(\partial T)
\end{equation}
for any $T\in \srR_{m}(\bbR^n)$ satisfying $\boldsymbol{\partial}(\clT - \clS(T)) = 0$.
\end{definition}

With this definition we can now say that by (\ref{eqn: boundary_S(T)_S(R)}), $\partial(\clS(T))=\clS(\partial T)$ for any $T\in \srF_m(\bbR^n)$.
\begin{proposition}
\label{prop: boundary_well_defined_on _scan_m}
Equation (\ref{eqn: defn_boundary_of_scan}) defines a boundary operator $\partial: \scan_m(\bbR^n)\rightarrow \scan_{m-1}(\bbR^n)$ with $\partial^2=0$, and $\partial: \scan_{m,K}(\bbR^n)\rightarrow \scan_{m-1,K}(\bbR^n)$ for $K$ compact.
\end{proposition}

\begin{proof}
For any $\clT\in \scan_{m}(\bbR^n)$, we first show that $\partial \clT := \clS(\partial T)$ is well-defined in the sense that $\clS(\partial T)$ is independent of the choices of $T\in \srR_{m,K}(\bbR^n)$ that satisfies $\boldsymbol{\partial}(\clT-\clS(T))=0$ and the choice of the compact set $K\subseteq \bbR^n$. 
Indeed, suppose there are two compact sets $K_1, K_2$ and an associated rectifiable current $T_j\in \srR_{m, K_j}(\bbR^n)\subseteq \srR_{m, K_1\cup K_2}(\bbR^n)$ with $\boldsymbol{\partial}(\clT-\clS(T_j))=0$ for $j=1,2$. 
Here, $\clT-\clS(T_j): \bfO^\ast(n, m)\times \bbR^m \to \srF_{0, K_j}(\bbR^n)\subseteq \srF_{0, K_1\cup K_2}(\bbR^n)$.
So, $\boldsymbol{\partial}(\clS(T_1-T_2))=\boldsymbol{\partial}(\clS(T_1)-\clS(T_2))=0$. 
By ~\cite[Proposition 3.1.6.]{de2003size}, $\partial(T_1-T_2)=0$ for $T_1-T_2\in \srR_{m, K_1\cup K_2}(\bbR^n)$. 
Therefore, $\partial T_1=\partial T_2$ and hence $\clS(\partial T_1)=\clS(\partial T_2)$.
This shows that the definition $\partial \clT := \clS(\partial T)$ is well-defined.
By definition, $\partial T\in \srF_{m-1, K}(\bbR^n)$. Thus by Lemma~\ref{lem: integral_flat_chains_are_scans}, $\partial \clT = \clS(\partial T)\in \scan_{m-1, K}(\bbR^n)$ and $\partial^2 \clT = \partial(\clS(\partial T))=\clS(\partial^2 T)=\clS(0)=0$.
\end{proof}

\begin{remark}
Recall that the symbol $\boldsymbol{\partial}$ is only used in the form of $\boldsymbol{\partial} \clT = 0$, denoting that $\clT$ is a cycle and that we cannot apply this symbol to other maps, while $\partial$ is a well-defined boundary operator on $\scan_m(\bbR^n)$. 
In $\scan_m(\bbR^n)$ the boundary operator $\partial$ is equivalent to $\boldsymbol{\partial}$, in the sense that for any $\clT \in \scan_m(\bbR^n)$, $\partial \clT = 0$ if and only if $\boldsymbol{\partial}\clT = 0$. 
\end{remark}

For each compact set $K\subseteq \bbR^n$, the following functional, denoted by $\newM_{H, K}$, serves as a modification of the $H$-mass.
Lemma~\ref{lem:FHequalsMH} shows that it agrees with $\bfM_H$ on the class of scans of rectifiable currents.

\begin{definition}\label{def: L_H}
For any compact $K\subseteq \bbR^n$, we define the functional $\newM_{H, K} : {\it Scan}_{m, K}(\bbR^n) \to \bbR$ by
\begin{equation}
\label{eqn:L_HK}
    \underline{\bfM}_{H, K}(\clT) := \inf\{\liminf_{i\to \infty} \bfM_H(T_i) : T_i \in \srR_{m,K}(\bbR^n), \clS(T_i) \to \clT,\text{ and } \partial(\clS( T_i))=\partial \clT\}.
\end{equation}
We will simply write $\newM_{H}(\clT)$ instead of $\newM_{H, K}(\clT)$ if $K$ is clear from context.
\end{definition}
\begin{remark}
In general, the value of $\newM_{H, K}(\clT)$ depends on the compact set $K$. 
Nevertheless, if $K_1 \subseteq K_2$ are nonempty compact subsets of $\bbR^n$ and $\clT \in \scan_{m, K_1}(\bbR^n) \cap \scan_{m, K_2}(\bbR^n)$, since $\srR_{m,K_1}(\bbR^n)\subseteq \srR_{m,K_2}(\bbR^n)$ and $\srF_{K_2}^H \leq \srF_{K_1}^H$, it follows that $\newM_{H, K_2}(\clT)\leq\newM_{H, K_1}(\clT)$.
\end{remark}

\begin{lem}\label{lem:FHequalsMH}
For any $T \in \srR_{m,K}(\bbR^n)$, $\newM_{H}(\clS(T)) = \bfM_H(T) = \bfM_H(\clS(T))$.
\end{lem}
\begin{proof}
Taking the constant sequence $T_i = T$ in (\ref{eqn:L_HK}) gives us that $\newM_{H}(\clS(T))\leq \bfM_H(T)$. 
We will now show that $\bfM_H(T) \leq \newM_{H}(\clS(T))$.
Let $T_i \in \srR_{m,K}(\bbR^n)$ with $\clS(T_i) \to \clS(T)$.
By definition, for $\boldsymbol{\theta}^\ast_{n, m}\times \clL^m$ almost every $(p, y)\in \bfO^\ast(n, m)\times \bbR^m$,
$\srF_K^H[\langle T_i, p, y\rangle - \langle T, p, y\rangle] \to 0 $ as $i\to \infty$.
By lower semi-continuity of $\bfM_H$ under $\srF_K^H$ convergence~\cite[Lemma 3.2.14.]{de2003size}, 
\[\bfM_H(\langle T, p, y\rangle) \leq \liminf_{i \to \infty} \bfM_H(\langle T_i, p, y\rangle).\]
Integrating both sides and using Fatou's lemma, we get 
\begin{eqnarray*}
    \bfM_H(\clS(T))&=&\frac{1}{ \boldsymbol{\beta}_1(m, n)}\int_{\bfO^\ast(n, m)\times \bbR^m} \bfM_H(\langle T, p, y\rangle)\, d(\mathbf{\theta}_{n, m}^\ast \times \clL^m)(p, y)\\
    &\leq&\frac{1}{ \boldsymbol{\beta}_1(m, n)}\int_{\bfO^\ast(n, m)\times \bbR^m}  \liminf_{i \to \infty} \bfM_H(\langle T_i, p, y\rangle)\, d(\mathbf{\theta}_{n, m}^\ast \times \clL^m)(p, y)\\
     &\leq&
     \liminf_{i \to \infty}\frac{1}{ \boldsymbol{\beta}_1(m, n)}\int_{\bfO^\ast(n, m)\times \bbR^m}   \bfM_H(\langle T_i, p, y\rangle)\, d(\mathbf{\theta}_{n, m}^\ast \times \clL^m)(p, y)\\
    &= & \liminf_{i \to \infty} \bfM_H(\clS(T_i)).
\end{eqnarray*}
Thus, $\bfM_H(\clS(T)) \leq \newM_{H}(T)$ and hence the result follows from $\bfM_H(\clS(T)) = \bfM_H(T)$.
\end{proof}

The following proposition says that $\newM_{H}$ is subadditive on $\scan_m(\bbR^n)$.
\begin{proposition}
\label{prop:LIsSubadditive}
For any nonempty compact subset $K$ of $\bbR^n$, 
$\newM_{H}(\clT_1 + \clT_2) \leq \newM_{H}(\clT_1) + \newM_{H}(\clT_2)$ for any $\clT_1, \clT _2 \in {\it Scan}_{m, K}(\bbR^n)$.
\end{proposition}
\begin{proof}
Let $K$ be a nonempty compact subset of $\bbR^n$ such that both $\clT_1, \clT_2 \in \scan_{m, K}(\bbR^n)$.
For each $k = 1, 2$, let $\{T^k_i\}_{i = 1}^\infty$ be a sequence in $\srR_{m, K}(\bbR^n)$ such that $\clS(T_i^k) \to \clT_k$ with $\partial(\clS(T_i^k))=\partial\clT_k$ for all $i \in \bbN$. By picking a subsequence if necessary, we may assume that
for each $k = 1, 2$, 
$\{\bfM_H(T^k_{i})\}_{i = 1}^\infty$ converges. 
Now consider the sequence $\{T_{i}^1 + T_{i}^2\}_{j = 1}^\infty$ and notice that $\clT_1 + \clT_2 \in \scan_{m, K}(\bbR^n)$, $\clS(T_{i}^1 + T_{i}^2) \to (\clT_1 + \clT_2)$, and $\partial (\clS(T_i^1 + T_i^2)) = \partial (\clT_1 + \clT_2)$. 
Indeed, for $\boldsymbol{\theta}^\ast_{n, m}\times \clL^m$ almost every $(p, y)\in \bfO^\ast(n, m)\times \bbR^m$, 
\[
\srF_K^H[(\clS(T_{i}^1 + T_{i}^2)(p, y) - (\clT_1 + \clT_2)(p, y)] \leq \srF_K^H[(\clS(T_{i}^1) - \clT_1)(p, y)] + \srF_K^H[(\clS(T_{i}^2) - \clT_2)(p, y)] \to 0
\]
as $i \to \infty$, and 
\[((\clT_1 + \clT_2) - \clS(T_{i}^1 + T_{i}^2))(p, y)(\boldsymbol{1}) = (\clT_1 - \clS(T_{i}^1)(p, y)(\boldsymbol{1}) + (\clT_2 - \clS(T_{i}^2)(p, y)(\boldsymbol{1}) = 0\]
for all $i \in \bbN$. 
By definition, this implies $\boldsymbol{\partial}((\clT_1 + \clT_2) - \clS(T_{i}^1 + T_{i}^2)) = 0$ for all $i \in \bbN$. 
Thus, 
\begin{align*}
\newM_{H}(\clT_1+\clT_2)&\leq
\liminf_{i \to \infty} \bfM_H(T^1_i + T^2_i) \leq \liminf_{i \to \infty} \bfM_H(T_i^1) + \bfM_H(T_i^2)\\
&= \lim_{i \to \infty} \bfM_H(T_i^1) + \lim_{i \to \infty} \bfM_H(T_i^2) = \liminf_{i \to \infty} \bfM_H(T_i^1) + \liminf_{i \to \infty} \bfM_H(T_i^2).
\end{align*}
Since the sequences $\{T^1_i\}_{i = 1}^\infty$ and $\{T_i^2\}_{i = 1}^\infty$ were arbitrary, we have $\newM_{H}(\clT_1+\clT_2)\leq \newM_{H}(\clT_1)+\newM_{H}(\clT_2)$.  
\end{proof}

\begin{lem}\label{lem:0dimWeakConvergence}
Let $K$ be a compact subset of $\bbR^n$ and let $T_i \in \bfI_0(\bbR^n)$ be a sequence of $0$ dimensional integral currents. 
If $\lim_{i\to \infty}\srF_K^H(T_i) = 0$, then $\lim_{i \to \infty} T_i(\boldsymbol{1}) = 0$.
\end{lem}
\begin{proof}
By definition of $\srF_K^H$, there exists a sequence $\{S_i\}_{i = 1}^\infty$ in $\srR_{1, K}(\bbR^n)$ such that $\lim_{i \to \infty}\bfM_H(T_i - \partial S_i) = 0$.
It is sufficient to prove that $\lim_{i\to\infty}\srF_K(T_i - \partial S_i) = 0$ since this implies that 
\[
\lim_{i \to \infty} T_i(\boldsymbol{1}) =
\lim_{i \to \infty} T_i(\boldsymbol{1}) - S_i(d(\boldsymbol{1})) =
\lim_{i \to \infty}(T_i - \partial S_i)(\boldsymbol{1}) = 0.
\]
To prove that $\lim_{i\to\infty}\srF_K(T_i - \partial S_i) = 0$ we will apply~\cite[Lemma 3.2.13.]{de2003size} which says that any sequence of integral flat chains $\{Q_i\}_{i=1}^\infty$ with $\sup_i \{\bfM(Q_i) + \bfM(\partial Q_i)\} <\infty$ and $\lim_{i\to\infty}\srF_K^H(Q_i) = 0$ will also have that $\lim_{i\to\infty}\srF_K(Q_i) = 0$.
To apply this lemma to our sequence $\{T_i - \partial S_i\}_{i = 1}^\infty$ of $0$ dimensional integral currents, we only need to show that $\sup_i \{\bfM(T_i - \partial S_i)\} < \infty$ and $\lim_{i\to \infty}\srF_K^H(T_i - \partial S_i) = 0$. 
Our last claim is shown by 
$
\srF_K^H(T_i - \partial S_i) \leq \bfM_H(T_i - \partial S_i) \to 0.
$
We may also assume that 
$
\sup_i \{\bfM_H(T_i - \partial S_i)\} < \infty
$. 
Remark 3.2.7.~in~\cite{de2003size} says that $H(\bfM(T)) \leq \bfM_H(T)$ for all $T \in \bfI_0(\bbR^n)$ and hence $\sup_i\{H(\bfM(T_i - \partial S_i))\} < \infty$.
Since $H$ is strictly increasing, we lastly get that $\sup_i\{\bfM(T_i - \partial S_i)\} < \infty$.
\end{proof}

Using this modified version of the $H$-mass $\newM_{H}$ on the class $\scan_{m,K}(\bbR^n)$ of $m$-dimensional scans in $K$ with boundary, we now show that the {\it Plateau's problem with $H$-mass}~(\ref{eq: HPP-scans}) has a solution for each $K$.

\begin{thm}\label{thm: plateausProblemWithHMass}
Let $K$ be a nonempty compact subset of $\bbR^n$.
For any $\clT_0 \in \scan_{m,K}(\bbR^n)$, there exists a scan $\clT^\ast \in \scan_{m,K}(\bbR^n)$ with $\partial \clT^\ast =\partial\clT_0$ such that 
\begin{equation}
    \newM_{H}(\clT^\ast) = \min \{\newM_{H}(\clT): \clT\in \scan_{m,K}(\bbR^n),\ \partial\clT =\partial \clT_0\}.\label{eq:MH-problem}
\end{equation}
Moreover, $\clT^\ast(p, y)\in\bfI_{0, K}(\bbR^n)$ for $\boldsymbol{\theta}^\ast_{n, m}\times \clL^m$ almost every $(p, y)\in \bfO^\ast(n, m)\times \bbR^m$.
\end{thm}

\begin{proof}
By assumption $\{\clT \in \scan_{m, K}(\bbR^n) : \partial \clT = \partial \clT_0\}$ contains $\clT_0$, and hence is non-empty. 
We may assume that $\inf \{\newM_{H}(\clT): \clT\in \scan_{m,K}(\bbR^n),\ \partial\clT =\partial \clT_0\}<\infty$ as otherwise $\newM_{H}(\clT_0)=\infty$ and we may simply pick $\clT^\ast=\clT_0$. 
Let $\{\clT^i\}_{i = 1}^\infty$ be an $\newM_{H}$ minimizing sequence in this collection. 
That is, $\clT^i \in \scan_{m,K}(\bbR^n)$, $\partial\clT^i =\partial \clT_0$, and 
\begin{equation}\label{eq:minimizingSequence}
\lim_{i\to \infty}\newM_{H}(\clT^i) = \inf \{\newM_{H}(\clT): \clT\in \scan_{m,K}(\bbR^n),\ \partial\clT =\partial \clT_0\}. 
\end{equation}
Since $\{\clT^i\}_{i = 1}^\infty$ is an $\newM_{H}$ minimizing sequence, without loss of generality, we may assume that 
\begin{equation}\label{eq:unifFH}
\sup_i \newM_{H}(\clT^i) \leq C < \infty
\end{equation}
for some $C > 0$. 
By definition of $\newM_{H}$ in (\ref{eqn:L_HK}), for each $i \in \bbN$, there exists a sequence $\{T_j^i\}_{j = 1}^\infty$ in $\srR_{m,K}(\bbR^n)$ such that
\begin{equation}\label{eq: limit-of-MH}
    \lim_{j \to \infty} \bfM_H(T_j^i) = \newM_{H}(\clT^i), \quad \text{ and } \quad
    \partial(\clS(T_j^i))=\partial \clT^i=\partial\clT_0
     \text{ for all } j\in \bbN.
\end{equation}
Since $\clT_0\in \scan_{m,K}(\bbR^n)$, there exists a $T_0\in \srR_{m,K}(\bbR^n)$ such that $\partial \clT_0=\partial(\clS(T_0))$.
By the proof of Proposition \ref{prop: boundary_well_defined_on _scan_m}, the condition $\partial(\clS(T_j^i))=\partial\clT_0$ imply that $\partial T_j^i=\partial T_0$ for all $i,j \in \bbN$.
By~(\ref{eq: limit-of-MH}), for each $i \in \bbN$ there exists an $N_i \in \bbN$ such that
\begin{equation}\label{eq:MHandFH}
|\bfM_H(T^i_{N_i})- \newM_{H}(\clT^i)|\le\frac{1}{i}.
\end{equation}
Hence by (\ref{eq:unifFH}), $\sup_i \bfM_H(T^i_{N_i}) \leq C + 1 < \infty$. 
Thus, 
\[\bfM_H(\clS(T^i_{N_i} - T_0)) = \bfM_H(T^i_{N_i} - T_0) \leq C + 1 + \bfM_H(T_0) \quad \text{ for all } i \in \bbN.\]
On the other hand, by \cite[Remark 3.3.3]{de2003size}, \cite[Lemma  3.1.2]{de2003size}, and the fact that $\partial T^i_{N_i} = \partial T_0$,
\begin{align*}
\bfD(\clS(T^i_{N_i} - T_0)) &\leq \bfc_1[\bfM_H((T^i_{N_i} - T_0)\times\llbracket \bfO^\ast(n, m)\rrbracket) + \bfM_H(\partial((T^i_{N_i} - T_0)\times\llbracket \bfO^\ast(n, m)\rrbracket))]\\
&=\bfc_1\bfM_H((T^i_{N_i} - T_0)\times\llbracket \bfO^\ast(n, m)\rrbracket)\\
&\leq\bfc_2 \bfM_H(T^i_{N_i} - T_0)\\
&\leq \bfc_2[C + 1 + \bfM_H(T_0)],
\end{align*}
where the constants $\bfc_1$ and $\bfc_2$ depend only on $n$ and $K$.
By Corollary \ref{thm:compactness} of the BV compactness theorem \cite[Theorem 3.4.1]{de2003size} of De Pauw and Hardt, and replacing $\{T_{N_i}^i\}$ by a subsequence if necessary, there exists a $\boldsymbol{\theta}^\ast_{n, m}\times \clL^m$ measurable map $\clR : \bfO^\ast(n, m)\times \bbR^m \to \bfI_{0, K}(\bbR^n)$ such that
\begin{equation}\label{eq:pointwiseConv}
\lim_{i\rightarrow \infty}\srF_K^H[\clS(T_{N_i}^i-T_0)(p, y) - \clR(p,y)]=0
\end{equation}
for $\boldsymbol{\theta}^\ast_{n, m}\times \clL^m$ almost every $(p, y)\in \bfO^\ast(n, m)\times \bbR^m$.

Now, let $\clT^\ast = \clR + \clS(T_0)$.
By definition,~(\ref{eq:pointwiseConv}) implies $\clS(T_{N_i}^i) \to \clR + \clS(T_0)= \clT^\ast$.
By Lemma~\ref{lem:0dimWeakConvergence}, this implies 
\[
(\clT^\ast - \clS(T_0))(p, y)(\boldsymbol{1}) = \lim_{i \to \infty}\langle T_{N_i}^i-T_0, p, y\rangle (\boldsymbol{1}) = 0
\]
$\boldsymbol{\theta}^\ast_{n, m}\times \clL^m$ almost every $(p, y)$, where the second equality is given by~\cite[Proposition~3.1.6.]{de2003size} since $\partial T_{N_i}^i = \partial T_0$ for all $i \in \bbN$.
This gives us our boundary condition, $\boldsymbol{\partial}(\clT^\ast - \clS(T_0)) = 0$, i.e., $\partial \clT^\ast=\partial(\clS(T_0))=\partial \clT_0$.
As a result, $\clT^\ast \in \scan_{m,K}(\bbR^n)$ with $\partial\clT =\partial \clT_0$.
Hence 
\[\newM_{H}(\clT^\ast) \geq \inf \{\newM_{H}(\clT): \clT\in \scan_{m,K}(\bbR^n),\ \partial\clT =\partial \clT_0\}.\]
Additionally, by~(\ref{eqn:L_HK}),~(\ref{eq:MHandFH}), and~(\ref{eq:minimizingSequence}),
\begin{align*}
\newM_{H}(\clT^\ast) \leq \liminf_{i\to \infty} \bfM_H(T_{N_i}^i) = \liminf_{i \to \infty} \newM_{H}(\clT^i) = \inf \{\newM_{H}(\clT): \clT\in \scan_{m,K}(\bbR^n),\ \partial\clT =\partial \clT_0\}.
\end{align*}
Therefore, $\clT^\ast$ satisfies (\ref{eq:MH-problem}) as desired.
\end{proof}

\begin{corollary}
Let $K$ be a nonempty compact and convex subset of $\bbR^n$, and $H$ be a concave integrand. 
For any $T_0 \in \srR_{m,K}(\bbR^n)$ and 
\begin{equation}
   \clT^\ast\in  \mathrm{argmin} \{\newM_{H}(\clT): \clT\in \scan_{m,K}(\bbR^n),\ \partial\clT =\partial \clT_0\}
\end{equation}
with $\clT_0:=\clS(T_0)$,
we have that $\newM_{H}(\clT^\ast)= \Gamma(\bfM_H, T_0, \bbR^n)$.
\end{corollary}

\begin{proof}
First, recall that $\Gamma(\bfM_H, T_0, \bbR^n) = \inf\{\bfM_H(T) : T \in \srR_{m, K}(\bbR^n),\ \partial T = \partial T_0\}$ since $\bfM_H({\pi_K}_\#(T))\leq \bfM_H(T)$ for any $T\in \srR_m(\bbR^n)$ under the nearest point projection map $\pi_K$.

Since $\clS_{m,K}(\bbR^n) \subseteq \scan_{m,K}(\bbR^n)$, by Lemma \ref{lem:FHequalsMH}, we have $\newM_{H}(\clT^\ast)\leq \inf\{\bfM_H(T) : T \in \srR_{m, K}(\bbR^n),\ \partial T = \partial T_0\} = \Gamma(\bfM_H, T_0, \mathbb{R}^n)$. 
On the other hand, since each $\bfM_H(T^i_j)\ge \inf\{\bfM_H(T) : T \in \srR_{m, K}(\bbR^n),\ \partial T = \partial T_0\}$, by (\ref{eq: limit-of-MH}), we have
$\newM_{H}(\clT^\ast)\geq \Gamma(\bfM_H, T_0, \mathbb{R}^n)$ as well. 
Therefore, $\newM_{H}(\clT^\ast)=\Gamma(\bfM_H, T_0, \mathbb{R}^n)$ as desired.
\end{proof}

\section{Partial Plateau's problem with $H$-mass}\label{sec: PPPwithHmass}

In this section, we assume that $\Phi \in \srD^{m - 1}(\bbR^n)$, $H$ is a concave integrand, and that $B \in \srR_{m-1}(\bbR^n)$ with $\bfM_H(B) < \infty$.

We are interested in what we call, {\it partial Plateau's problem with $H$-mass}:
\begin{equation}\label{partialplateausproblem}
\inf\{\bfE_{H, \Phi}(T) : T \in \srR_{m}(\bbR^n), \partial T \preceq B\}\quad \mathrm{ where }\quad \bfE_{H, \Phi}(T) := \bfM_H(T) - \partial T(\Phi).
\end{equation}

 To consider this problem, as in the previous section, we will consider an extension of $\bfE_{H, \Phi}$ to elements in $\scan_{m, K}(\bbR^n)$ for a fixed compact set $K \subseteq \bbR^n$.
 
 For any $\clT\in \scan_{m, K}(\bbR^n)$, by definition, there exists a $T\in \srR_{m, K}(\bbR^n)$ such that $\partial \clT=\clS(\partial T)$.
We simply write $\partial\clT\preceq \clS(B)$ whenever $\partial T\preceq B$. 
Also, we define $\partial \clT(\Phi):=\partial T(\Phi)$ and
\[\bfE_{H, \Phi}(\clT) := \newM_{H}(\clT) - \partial \clT(\Phi).\] 
 It extends the definition of $\bfE_{H, \Phi}$ in (\ref{partialplateausproblem}) since by Lemma \ref{lem:FHequalsMH}, it holds that $\bfE_{H, \Phi}(\clS(T))=\bfE_{H, \Phi}(T) $ for any $T\in \srR_{m,K}(\bbR^n)$.
\begin{thm}
\label{thm: mainthm}
Let $K$ be a nonempty, compact, and convex subset of $\bbR^n$, $\Phi \in \srD^{m - 1}(\bbR^n)$, $H$ a concave integrand, and let $B \in \srR_{m-1}(\bbR^n)$ with
$\bfM_H(B) < \infty$. 
For any $\clT_0\in \scan_{m, K}(\bbR^n)$, there exists a scan $\clT^\ast:\bfO^\ast(n, m)\times \bbR^m\to \bfI_{0, K}(\bbR^n)$ such that $\clT^\ast \in \scan_{m, K}(\bbR^n)$,  
$\partial(\clT^\ast-\clT_0) \preceq \clS(B)$, and
\begin{equation}
     \bfE_{H, \Phi}(\clT^\ast) = \min\{\bfE_{H, \Phi}(\clT) : \clT \in \scan_{m,K}(\bbR^n),\ \partial(\clT-\clT_0) \preceq \clS(B)\}.
\label{eq:partial-MH-problem}
\end{equation}
Moreover, if $\clT_0=\clS(T_0)$ for some $T_0\in \srR_{m,K}(\bbR^n)$, then
\begin{equation}
\bfE_{H, \Phi}(\clT^\ast) =\inf\{\bfE_{H, \Phi}(T) : T \in \srR_m(\bbR^n), \partial (T-T_0) \preceq B\}.
\label{eq: T^ast_minimizer}
\end{equation}
\end{thm}

The proof of this theorem will make use of the following two lemmas.

\begin{lem}\label{lem:subcurrentRepresentation}
Suppose that $B \in \srR_{k}(\bbR^n)$ for $k \geq 0$. 
If $A \preceq B$, then $A$ is a $k$ dimensional real rectifiable current with $\bfM_H(A) \leq \bfM_H(B)$. 
\end{lem}
\begin{proof}
Since $B \in \srR_k(\bbR^n)$, 
\[B = \clH^k\myell M\wedge \theta\xi\]
such that $M$ is an $(\clH^k, k)$-rectifiable subset of $\bbR^n$, $\theta : M \to \bbZ$ is summable, and $\xi : M \to \wedge_k(\bbR^n)$ is a simple $k$-vectorfield on $M$ that coincides with the approximate tangent space of $M$.
By~\cite[\text{Lemma 3.7}]{paolini2006optimal}, $A \preceq B$ implies that
\[A = \clH^m\myell M \wedge\sigma\theta\xi \text{ for some Borel function } \sigma : \bbR^n \to [0, 1].\] 
Thus,
\begin{equation}
\bfM_H(A) = \int_M H(\sigma(x)\theta(x))\,d\clH^m(x)\leq \int_M H(\theta(x))\,d\clH^m(x)= \bfM_H(B)
\end{equation}
since $0 \leq \sigma \leq 1$ and $H$ is monotonically non-decreasing.
\end{proof}

\begin{lem}\label{lemma:fillingconvergence}
Let $m \in \bbN$, $K$ be a nonempty, convex, compact subset of $\bbR^n$, $Z \in \srR_{m-1, K}(\bbR^n)$ with $\partial Z = 0$, and let $\{Z_j\}_{j=1}^\infty \subseteq \srR_{m-1, K}(\bbR^n)$ with $\partial Z_j = 0$ for all $j \in \bbN$. 
If $\lim_{j\to \infty} \srF_K(Z_j - Z) = 0$, then there exists $C, C_1, C_2, \dots$ in $\srR_{m, K}(\bbR^n)$ such that $\partial C = Z$, $\partial C_j = Z_j$ for all $j \in \bbN$, and $\lim_{j\to \infty}\srF_K(C_j - C) = 0$.
\end{lem}
\begin{proof}
For each $j\in \bbN$, let $C_j=\boldsymbol{\delta}_v\join Z_j$ and $C=\boldsymbol{\delta}_v\join Z$ be the cones over $Z_j$ and $Z$ respectively with the fixed vertex $v \in K$~\cite[4.1.11]{federer-1960-1}. 
Note that since $K$ is convex, the cones $C_j$ and $C$ are in $\srR_{m, K}(\bbR^n)$.
By the homotopy formula,
\[\partial C_j=\partial(\boldsymbol{\delta}_v\join Z_j)=Z_j-\boldsymbol{\delta}_v\join \partial Z_j=Z_j
\text{ and } 
\partial C=\partial(\boldsymbol{\delta}_v\join Z)=Z-\boldsymbol{\delta}_v\join\partial Z=Z.\]
For any $\epsilon>0$, since $\lim_{j\to \infty} \srF_K(Z_j - Z) = 0$, there exists an $N\in \mathbb{N}$ such that whenever $j\ge N$, it holds that
$\srF_K(Z_j - Z)<\epsilon$.
Thus, by the definition of the flat norm $\srF_K$, there exists $R_j$ and $S_j$ such that 
\[Z_j-Z=R_j+\partial S_j, \text{ and } \mathbf{M}(R_j) +\mathbf{M}(S_j)<\epsilon.\]
Now,
\begin{align*}
C_j - C & = \boldsymbol{\delta}_v\join Z_j - \boldsymbol{\delta}_v\join Z = \boldsymbol{\delta}_v\join (Z_j-Z)=\boldsymbol{\delta}_v\join (R_j+\partial S_j)\\ &=\boldsymbol{\delta}_v\join R_j+\boldsymbol{\delta}_v\join\partial S_j
=\boldsymbol{\delta}_v\join R_j + S_j-\partial (\boldsymbol{\delta}_v\join S_j),
\end{align*}
where we used the homotopy formula in the last equality. 
Assume that the compact set $K$ is contained in the ball $B(0, \rho)$ for some $\rho>0$. 
Then,
\[\mathbf{M}(\boldsymbol{\delta}_v\join R_j)\le \rho \mathbf{M}(R_j)\text{ and } \mathbf{M}(\boldsymbol{\delta}_v\join S_j)\le \rho \mathbf{M}(S_j),\]
and thus
\[\srF_K(C_j - C)\le \mathbf{M}(\boldsymbol{\delta}_v\join R_j+S_j)+\mathbf{M}(\boldsymbol{\delta}_v\join S_j)\le \rho \mathbf{M}(R_j)+\mathbf{M}(S_j)+\rho \mathbf{M}(S_j) \le (1+\rho)\epsilon.\]
This shows that $\lim_{j\to \infty} \srF_K(C_j - C) = 0$ as desired.
\end{proof}

\begin{remark}
The following proof of Theorem \ref{thm: mainthm} is similar to the one of Theorem \ref{thm: plateausProblemWithHMass}. Nevertheless, as the boundary of the minimizing sequence is not fixed, one can not apply directly Theorem \ref{thm: plateausProblemWithHMass} here.
\end{remark}
\begin{proof}[Proof of Theorem \ref{thm: mainthm}]
Since $\clT_0\in\scan_{m,K}(\bbR^n)$, by definition, there exists a $T_0\in \srR_{m,K}(\bbR^n)$  with $\partial \clT_0=\partial(\clS(T_0))$.
Since the zero-current ${\bf 0}\preceq B$, both $\clT_0$ and $\clS(T_0)$ are contained in the set $\{ \clT \in \scan_{m, K}(\bbR^n) : \partial(\clT-\clT_0) \preceq \clS(B)\}$. 
Let $\{\clT_i\}_{i=1}^\infty$ be an $\bfE_{H, \Phi}$ minimizing sequence in this collection.
That is, $\clT_i \in \scan_{m, K}(\bbR^n)$ with $\partial(\clT_i-\clT_0) \preceq \clS(B)$ for each $i = 1, 2, \dots$, and 
\begin{equation}
    \label{eq:E-minimizingSequence}
\lim_{i \to \infty} \bfE_{H, \Phi}(\clT_i) = \inf\{\bfE_{H, \Phi}(\clT) : \clT \in \scan_{m, K}(\bbR^n),\ \partial(\clT_i-\clT_0) \preceq \clS(B)\}.
\end{equation}
Without loss of generality, we may assume that $\bfE_{H, \Phi}(\clT_i) \leq \bfE_{H, \Phi}(\clS(T_0))=\bfE_{H, \Phi}(T_0)$ for all $i \in \bbN$. 
Note that by definition, our boundary condition $\partial(\clT_i - \clT_0) \preceq \clS(B)$ implies that for each $i \in \bbN$, there exists a $T_i \in \srR_{m,K}(\bbR^n)$ such that $\partial\clT_i =\partial (\clS(T_i))$ and $\partial (T_i-T_0) \preceq B$. 
Note also that by the {\it boundary rectifiability theorem}~\cite{simon-1984-lectures}, $\partial(T_i - T_0) \in \srR_{m-1}(\bbR^n)$ and hence $T_i - T_0$ are integral currents.
Now, 
\begin{align}
\newM_{H}(\clT_i) &= \bfE_{H, \Phi}(\clT_i) + \partial T_i (\Phi)\leq \bfE_{H, \Phi}(T_0) + \partial T_i (\Phi)= \bfM_H(T_0)+\partial (T_i-T_0) (\Phi) \\
&\leq \bfM_H(T_0)+\bfM(\partial (T_i-T_0))||\Phi||\leq \bfM_H(T_0)+\bfM(B)||\Phi|| < \infty,\label{eq:unifLHbound}
\end{align}
since $\bfM_H(T_0)\leq H(2)\bfM(T_0) < \infty$ by using the fact that $H(\theta) \leq H(2) \theta$ for all $\theta \geq 1$ in equation~(\ref{def:HMass}). 
By definition of $\newM_{H}$ in (\ref{eqn:L_HK}), for each $i \in \bbN$, there exists a sequence $\{T_j^i\}_{j = 1}^\infty$ in $\srR_{m,K}(\bbR^n)$ such that
\begin{equation}\label{eq: limit-of-MHT0}
    \lim_{j \to \infty} \bfM_H(T_j^i) = \newM_{H}(\clT_i), \quad \text{ and } \quad \partial (\clS(T_j^i))=\partial\clT_i \text{ for all } j\in \bbN.
\end{equation}
Notice that since $\partial (\clS(T_j^i))=\partial\clT_i=\partial (\clS(T_i))$ it follows that  $\partial T^i_j = \partial T_i$ for all $i,j \in \bbN$.
By~(\ref{eq: limit-of-MHT0}), for each $i \in \bbN$ there exists an $N_i \in \bbN$ such that
\begin{equation}\label{eq:MHandFH_partial_2}
|\bfM_H(T^i_{N_i})- \newM_{H}(\clT_i)|\le\frac{1}{i}.
\end{equation}
Hence by (\ref{eq:unifLHbound}) and  (\ref{eq:MHandFH_partial_2}), for all  $i \in \bbN$,
\begin{align*}
    \bfM_H(\clS(T^i_{N_i} - T_0)) &= \bfM_H(T^i_{N_i} - T_0)\leq \bfM_H(T^i_{N_i})+\bfM_H(T_0)\\
&\leq \newM_{H}(\clT_i)+1+\bfM_H(T_0)\leq 2\bfM_H(T_0)+\bfM(B)||\Phi|| + 1.
\end{align*}
Furthermore, by \cite[Remark 3.3.3]{de2003size}, \cite[Lemma  3.1.2]{de2003size}, Lemma \ref{lem:subcurrentRepresentation} and the fact that $\partial T^i_{N_i} = \partial T_i$, it holds that
\begin{align*}
\bfD(\clS(T^i_{N_i} - T_0)) &\leq \bfc_1[\bfM_H((T^i_{N_i} - T_0)\times\llbracket \bfO^\ast(n, m)\rrbracket) + \bfM_H(\partial((T^i_{N_i} - T_0)\times\llbracket \bfO^\ast(n, m)\rrbracket))]\\
&=\bfc_1[\bfM_H((T^i_{N_i} - T_0)\times\llbracket \bfO^\ast(n, m)\rrbracket)+ \bfM_H(\partial((T_i - T_0)\times\llbracket \bfO^\ast(n, m)\rrbracket))]\\
&\leq\bfc_2[ \bfM_H(T^i_{N_i} - T_0)+\bfM_H(\partial(T_i - T_0))]\\ 
&\leq \bfc_2[2\bfM_H(T_0)+\bfM(B)||\Phi|| + 1 +\bfM_H(B)],
\end{align*}
where the constants $\bfc_1$ and $\bfc_2$ depend only on $n$ and $K$. Thus, for each $i$,
\begin{align*}
    &\bfM_H(\clS(T^i_{N_i} - T_0))+\bfD(\clS(T^i_{N_i} - T_0))\\
    &\leq (1+\bfc_2)(2\bfM_H(T_0)+\bfM(B)||\Phi|| + 1)+\bfc_2\bfM_H(B)<\infty.
\end{align*}

Now, by Corollary \ref{thm:compactness} and replacing $\{T_{N_i}^i\}$ by a subsequence if necessary, we may assume that
\begin{equation}\label{eq:pointwiseConv_partial}
\lim_{i\rightarrow \infty}\srF_K^H[\clS(T_{N_i}^i-T_0)(p, y) - \clR(p,y)]=0
\end{equation}
for $\boldsymbol{\theta}^\ast_{n, m}\times \clL^m$ almost every $(p, y)\in \bfO^\ast(n, m)\times \bbR^m$, where $\clR : \bfO^\ast(n, m)\times \bbR^m \to \bfI_{0, K}(\bbR^n)$ is a measurable map. Let $\clT^\ast = \clR + \clS(T_0)$. By definition,~(\ref{eq:pointwiseConv_partial}) implies $\clS(T_{N_i}^i) \to \clT^\ast$. By Lemma~\ref{lem:0dimWeakConvergence}, this implies for a.e.~$(p, y)$,
\begin{equation}
\label{eq: T^ast-minus-T_i}  
\lim_{i \to \infty} (\clT^\ast - \clS(T_i))(p, y)(\boldsymbol{1}) =\lim_{i \to \infty} (\clT^\ast - \clS(T_{N_i}^i))(p, y)(\boldsymbol{1}) + \lim_{i \to \infty}\langle T_{N_i}^i-T_i, p, y\rangle (\boldsymbol{1}) = 0,
\end{equation}
 where the second equality is given by~\cite[Proposition~3.1.6.]{de2003size} since $\partial T_{N_i}^i = \partial T_i$ for all $i \in \bbN$.
 
 Our next goal is to show that $\clT^\ast\in \scan_{m,K}(\bbR^n)$ and satisfy the boundary condition $\partial(\clT^\ast-\clT_0) \preceq \clS(B)$. Since $\partial(\clT^\ast-\clT_0)=\partial(\clT^\ast-\clS(T_0))=\partial\clR$, it is sufficient to show that $\partial\clR=\partial (\clS(C))$ for some $C\in \srR_{m, K}(\bbR^n)$ with $\partial C \preceq B$. Since $\partial (T_i-T_0) \preceq B$, we have $\bfM(\partial (T_i-T_0))\leq \bfM(B)<\infty$ for each $i$.
By the compactness theorem of rectifiable currents and replacing $(T_i-T_0)$ by a subsequence if necessary, we may assume that
$\lim_{i\rightarrow \infty}\srF_K(\partial (T_i-T_0) - S)=0$
for some $S\in \srR_{m-1}(\bbR^n)$. Since $\partial (T_i-T_0) \preceq B$, by Lemma~\ref{lemma:subcurrentandconvergence}, $S \preceq B$.

Let $C_i=\boldsymbol{\delta}_v \join \partial (T_i-T_0)$ and $C=\boldsymbol{\delta}_v \join S$ be the cones with a fixed cone vertex $v \in K$. 
Note that both $C_i$ and $C$ are in $\srR_{m, K}(\bbR^n)$ since $K$ is convex. 
Then, $\partial C_i=\partial (T_i-T_0)$ and by Lemma \ref{lemma:fillingconvergence}, $\displaystyle\lim_{i\to \infty}\srF_K(C_i - C) = 0$.
By the integral-geometric equality,
\[\int_{\bfO^\ast(n, m)\times \bbR^m} \srF_K(\langle C_{i} - C, p, y\rangle )\, d\left(\boldsymbol{\theta}^\ast_{n, m} \times \clL^m\right) (p, y) \leq \boldsymbol{\beta}_1(n, m)\srF_K(C_{i} - C) \to 0.\]
Thus, for $\boldsymbol{\theta}^\ast_{n, m} \times \clL^m$-a.e. $(p, y) \in \bfO^\ast(n, m)\times \bbR^m$, $\lim_{i \to \infty} \srF_K(\langle C_{i}-C, p, y\rangle) = 0$.
As a result, for $\boldsymbol{\theta}^\ast_{n, m} \times \clL^m$-a.e. $(p, y) \in \bfO^\ast(n, m)\times \bbR^m$, by (\ref{eq: T^ast-minus-T_i}),
\begin{align*}
(\clR - \clS(C))(p, y)(\boldsymbol{1})
&=(\clT^\ast -\clS(T_0)- \clS(C))(p, y)(\boldsymbol{1})\\
&=\lim_{i \to \infty} (\clT^\ast - \clS(T_i))(p, y)(\boldsymbol{1})+ \lim_{i \to \infty} \langle T_{i}-T_0 - C_{i}, p, y\rangle (\boldsymbol{1})= 0,
\end{align*}
because $\partial C_i=\partial (T_i-T_0)$.
This shows that $\partial(\clR - \clS(C)) = 0$. Since $\partial C=S \preceq B$, we have $\partial\clR=\clS(\partial C) \preceq S(B)$.
This gives us our boundary condition, $\partial(\clT^\ast - \clS(T_0+C)) =\partial(\clR - \clS(C))= 0$.
Since $\clS(T_{N_i}^i) \to \clT^\ast$, it follows that $\clT^\ast \in \scan_{m,K}(\bbR^n)$
with $\partial(\clT^\ast-\clS(T_0)) =\partial\clR \preceq B$.
Hence $\clT^\ast$ belongs to the collection $\{\clT \in \scan_{m,K}(\bbR^n) : \partial(\clT-\clS(T_0)) \preceq \clS(B)\}$.
Additionally, by~(\ref{eqn:L_HK}),~(\ref{eq:MHandFH_partial_2}), and~(\ref{eq:E-minimizingSequence}),
\begin{align*}
\bfE_{H, \Phi}(\clT^*)&=
\newM_{H}(\clT^\ast) - \partial (T_0+C)(\Phi)\\
&\leq \liminf_{i\to \infty} \bfM_H(T_{N_i}^i)  -\lim_{i \to \infty} \partial (T_0+C_i)(\Phi)\\
&= \liminf_{i \to \infty} \newM_{H}(\clT_i) -\lim_{i \to \infty}\partial T_i(\Phi)\\
&=\liminf_{i \to \infty} \bfE_{H, \Phi}(\clT_i)\\
&= \inf\{\bfE_{H, \Phi}(\clT) : \clT \in \scan_{m,K}(\bbR^n),\ \partial(\clT-\clS(T_0)) \preceq \clS(B)\}.
\end{align*}
Therefore, $\clT^\ast$ is a desired $\bfE_{H, \Phi}$ minimizer satisfying (\ref{eq:partial-MH-problem}). 

Now, suppose that $\clT_0 = \clS(T_0)$ for some $T_0 \in \srR_{m, K}(\bbR^n)$. 
We prove (\ref{eq: T^ast_minimizer}) as follows. Since $\{\clS(T) : T \in \srR_{m,K}(\bbR^n)\} \subseteq \scan_{m,K}(\bbR^n)$, by Lemma \ref{lem:FHequalsMH}, we have
\begin{align*}
     \bfE_{H, \Phi}(\clT^*) &\leq \inf\{\bfE_{H, \Phi}(T) : T \in \srR_{m,K}(\bbR^n), \partial (T-T_0) \preceq B\}\\
&= \inf\{\bfE_{H, \Phi}(T) : T \in \srR_m(\bbR^n), \partial (T-T_0) \preceq B\}\leq \bfE_{H, \Phi}(T^i_{N_i}) 
\end{align*}
for each $i$. Taking $i\rightarrow \infty$,
by (\ref{eq: limit-of-MHT0}), we have the desired equality (\ref{eq: T^ast_minimizer}).
\end{proof}

In particular, when $T_0=0$, we have the following corollary.

\begin{corollary}\label{cor: main}
Let $K$ be a nonempty, compact, and convex subset of $\bbR^n$, $\Phi \in \srD^{m - 1}(\bbR^n)$, $H$ a concave integrand, and let $B \in \srR_{m-1}(\bbR^n)$ with
$\bfM_H(B) < \infty$.
Then there exists $\clT^\ast:\bfO^\ast(n, m)\times \bbR^m\to \bfI_{0, K}(\bbR^n)$ such that $\clT^\ast \in \scan_{m, K}(\bbR^n)$, 
$\partial\clT^\ast \preceq \clS(B)$, and
\begin{align*}
      \bfE_{H, \Phi}(\clT^*) &= \min\{\bfE_{H, \Phi}(\clT) : \clT \in \scan_{m,K}(\bbR^n), \partial\clT \preceq \clS(B)\}\\
& = \inf\{\bfE_{H, \Phi}(T) : T \in \srR_m(\bbR^n), \partial T \preceq B\}.
\end{align*}
\end{corollary}

\bibliographystyle{plain}
\bibliography{bibliography.bib}
\end{document}